\documentclass[a4paper,english, 11pt]{amsart}
\makeatletter
 \theoremstyle{plain}
\newtheorem{thm}{Theorem}[subsection]
\theoremstyle{plain}
  \newtheorem{prop}[thm]{Proposition}
\theoremstyle{plain}

\theoremstyle{plain}
 \newtheorem{lemma}[thm]{Lemma}
\theoremstyle{plain}

\theoremstyle{plain}
\newtheorem{cor}[thm]{Corollary}
\theoremstyle{definition}
  \newtheorem{defn}[thm]{Definition}
 \theoremstyle{definition}

\theoremstyle{remark}
\newtheorem{rmk}[thm]{Remark}
\numberwithin{equation}{section}
 
\usepackage{amsmath}
\usepackage{amssymb}
\usepackage{amscd}
\usepackage{amsthm}
\usepackage{array}
\usepackage[arrow,matrix,tips]{xy}

\usepackage{stmaryrd}
\usepackage{xcolor}
\usepackage{multicol}
\usepackage{enumitem}
\usepackage{hyperref}

\pdfpageheight\paperheight
\pdfpagewidth\paperwidth
\newcommand{\Q}{\mathbb{Q}}

\newcommand{\Qp}{\mathbb{Q}_p}

\newcommand{\F}{\mathbb{F}}

\newcommand{\fM}{\mathfrak{M}}

\newcommand{\fS}{\mathfrak{S}}

\newcommand{\fm}{\mathfrak{m}}

\newcommand{\cE}{\mathcal{E}}

\newcommand{\cM}{\mathcal{M}}

\newcommand{\cO}{\mathcal{O}}

\newcommand{\phz}{\varphi}

\newcommand{\Gal}{\mathrm{Gal}}

\newcommand{\Ind}{\mathrm{Ind}}

\newcommand{\GL}{\mathrm{GL}}

\DeclareMathOperator{\et}{\acute et}
\DeclareMathOperator{\wa}{wa}

\DeclareMathOperator{\Mod}{Mod}

\DeclareMathOperator{\Fil}{Fil}

\DeclareMathOperator{\cris}{cris}

\newcommand{\ra}{\rightarrow}

\newcommand{\iarrow}{\hookrightarrow}

\newenvironment{smallpmatrix}
  {\left(\begin{smallmatrix}}
  {\end{smallmatrix}\right)}

\makeatother
\title[Reductions of some crystalline representations]{Reductions of some two-dimensional crystalline representations via Kisin modules}

\author{John Bergdall}
\address{Bryn Mawr College,
Department of Mathematics,
101 North Merion Avenue,
Bryn Mawr, PA 19010, USA}
\email{jbergdall@brynmawr.edu}

\author{Brandon Levin}
\address{Department of Mathematics,
University of Arizona, 
617 N Santa Rita Avenue, 
Tucson, Arizona 85721, USA}
\email{bwlevin@math.arizona.edu}

 
\begin{document}
\begin{abstract}
We determine rational Kisin modules associated with two-dimensional, irreducible, crystalline representations of $\Gal(\overline{\Q}_p/\Q_p)$ of Hodge-Tate weights $0, k-1$. If the slope is larger than $\lfloor \frac{k-1}{p} \rfloor$, we further identify an integral Kisin module, which we use to calculate the semisimple reduction of the Galois representation. In that range, we find that the reduction is constant, thereby improving on a theorem of Berger, Li, and Zhu.
\end{abstract}
\maketitle
\setcounter{tocdepth}{1}
\tableofcontents

\section{Introduction}
Let $p$ be a prime number and $\overline{\mathbb Q}_p$ be an algebraic closure of the $p$-adic numbers $\mathbb{Q}_p$. The aim of this paper is to study two-dimensional, irreducible, crystalline representations of $G_{\mathbb Q_p} = \Gal(\overline{\mathbb Q}_p/\mathbb Q_p)$ and their reductions modulo $p$.  Examples of such representations arise in the arithmetic of modular forms. Fontaine first calculated the corresponding reductions in the late 1970's for modular forms whose weights are small with respect to $p$. (The proof was never published; Edixhoven provided a proof in \cite{Edixhoven-Weights}.) Spurred on by the $p$-adic local Langlands correspondence for $\mathrm{GL}_2(\mathbb Q_p)$ there has recently been considerable attention paid to local questions, often without qualification on weights.

\subsection{Main result}

To make our discussion precise, write $v_p$ for the valuation on $\overline{\mathbb Q}_p$ normalized by $v_p(p) = 1$. Then, for each $k \geq 2$ and each $a_p \in \overline{\mathbb Q}_p$ satisfying $v_p(a_p) > 0$, there exists a unique two-dimensional, irreducible, crystalline representation $V_{k,a_p}$ whose Hodge--Tate weights are $0$ and $k-1$ and such that the characteristic polynomial of the crystalline Frobenius is $X^2 - a_p X + p^{k-1}$. Up to one-dimensional twists, these are all the two-dimensional, irreducible, crystalline representations of $G_{\mathbb Q_p}$. So, calculating the reductions in general reduces to the two-parameter family $V_{k,a_p}$. 

Let $\overline{V}_{k,a_p}$ be the semisimple reduction modulo $p$ of $V_{k,a_p}$. For $k$ fixed, it is known that $a_p \mapsto \overline{V}_{k,a_p}$ is locally constant (see \cite{Berger-LocalConstancy}, for example). So, focusing near to $a_p = 0$, there exists a smallest real number $\delta_p(k)$ for which $\overline{V}_{k,a_p}\cong \overline{V}_{k,0}$ whenever $v_p(a_p) > \delta_p(k)$. In terms of controlling $\delta_p(k)$,  Berger, Li, and Zhu proved fifteen years ago that $\delta_p(k) \leq \lfloor{k-2\over p-1}\rfloor$ (\cite{BergerLiZhu-SmallSlopes}). Our main theorem improves that result:
\begin{thm}[{Corollary \ref{cor:reduction}}]\label{thm:intro} Let $k \geq 2$. Then, $\overline{V}_{k,a_p} \cong \overline{V}_{k,0}$ for all $v_p(a_p) > \lfloor \frac{k-1}{p} \rfloor$.
\end{thm} 

This theorem advances our understanding of $\overline{V}_{k,a_p}$ when $v_p(a_p) \gg 0$. It complements many papers focusing on small $v_p(a_p)$ (\cite{BuzzardGee-SmallSlope, BuzzardGee-SmallSlope2, BhattacharyaGhate-Slope12, BhattacharyaGhateRozensztajn-Slope1, Arsovski-Slopes, NagelPande-Slopes23, GhateRai-Slopes3/2}).  Those works employ a strategy, pioneered by Buzzard and Gee, that leverages the $p$-adic local Langlands correspondence. By contrast, the earlier work of Berger--Li--Zhu uses Wach modules, which more directly determine lattices in crystalline Galois representations. Our approach belongs to that tradition, though we replace Wach modules with another tool from integral $p$-adic Hodge theory:\ Kisin modules.

Despite their theoretical importance, there are few examples of explicit calculations with Kisin modules like we give here. Those that do exist are recent and limited to small Hodge--Tate weights (\cite{CarusoDavidMezard-Calculation, LLLM-ShapesShadows, LLLM-WeightElimination}). One advantage of Kisin modules is their availability beyond two-dimensional representations of $G_{\Q_p}$, unlike approaches via $p$-adic local Langlands  (see the generalizations of \cite{BergerLiZhu-SmallSlopes} in \cite{Dousmanis-Reductions,YamashitaYasuda-Applications}), and their availability beyond crystalline situations, unlike Wach modules (cf.\ \cite{CarusoLubicz-Reductions}). For instance, the method outlined below was recently applied by the authors and Tong Liu in order to calculate reductions of some semi-stable, non-crystalline, representations of $G_{\Q_p}$ (\cite{BergdallLevinLiu-Semistable}).

Finally, Theorem \ref{thm:intro} can be improved. Computational evidence (\cite{Rozensztajn-2dlocus}) and global considerations (\cite{Gouvea-WhereSlopesAre,BuzzardGee-Slopes}) suggest that $\delta_p(k) \leq \lfloor{k-1\over p+1}\rfloor$, though precise predictions of local constancy phenomena related to Galois representations and modular forms have been wrong before (cf.\ \cite{BuzzardCalegari-GouveaMazur}). After the release of this article, Arsovski (\cite{Arsovski-Supersinglarities}) provided further evidence that $\delta_p(k) \leq \lfloor{k-1\over p+1}\rfloor$ by showing  $\delta_p(k) \leq \lfloor{{k-1\over p+1}\rfloor} + \lfloor{\log_p(k)\rfloor}$ as long as $p > 3$ and $k \not\equiv 1 \bmod p+1$. Arsovski uses a $p$-adic local Langlands approach, so they do not recover neither the more specific Theorem \ref{thm:explicit-descent-intro} below nor Theorem \ref{thm:descent2}, which applies to any $a_p$.

\subsection{Method}
The rest of the introduction is devoted to summarizing our method. We write $F$ for a finite extension of $\Qp$, $\Lambda$ for its ring of integers, and $\fm_F$ for the maximal ideal of $\Lambda$. The field $F$ will play the role of linear coefficients. Write $E(u) = u+p$. Define $\cO_F \subset F[\![u]\!]$ to be the subring of series converging on the disc $|u|_p<1$. We will consider $\varphi$-modules over $\mathcal O_F$ and $\fS_\Lambda = \Lambda[\![u]\!]$.  A finite height $\phz$-module over $\cO_F$ is a finite free $\cO_F$-module $\cM$ equipped with an operator $\varphi: \cM \rightarrow \cM$, called a Frobenius, that is semi-linear for $u \mapsto u^p$ on $\cO_F$ and for which the cokernel of the  linearization $\varphi^{\ast}\mathcal M \rightarrow \mathcal M$ is annihilated by $E^h$, for some non-negative integer. (We say $\mathcal M$ has height $\leq h$).  A Kisin module is a $\phz$-module over $\fS_{\Lambda}$ satisfying the same height condition.  We regularly describe a $\phz$-module (or Kisin module) by fixing a basis $\{e_i\}$ of $\mathcal M$ and giving the matrix $C$ of $\varphi$ in that basis.

Now let $k \geq 2$ and $a_p \in \fm_F$. By \cite{Kisin-FCrystals}, one may associate to  $V_{k,a_p}$ a unique $\phz$-module $\cM_{k,a_p}$ over $\mathcal O_F$ with height $\leq k-1$. More precisely, $\cM_{k, a_p}$ is constructed from the (contravariant) weakly-admissible filtered $\varphi$-module $D_{\cris}^{\ast}(V_{k,a_p})$. By the general theory, one may descend $\cM_{k,a_p}$ to a Kisin module $\fM_{k,a_p}$ and, though $\fM_{k,a_p}$ depends on a Galois stable lattice in $V_{k, a_p}$, the mod $p$ Galois representation $\overline V_{k,a_p}$ is completely determined by the $\varphi$-module $\fM_{k,a_p}/\fm_F\fM_{k,a_p}$. In this way, Kisin modules provide a theoretical tool for calculating $\overline{V}_{k,a_p}$. Unfortunately, both the passage from filtered $\varphi$-modules to finite height $\phz$-modules over $\cO_F$ and the descent to $\fS_\Lambda$ are difficult to navigate from the point of view of direct calculation, except in {\em very} special circumstances.

Suppose, however, that we have defined a rank two Kisin module $\fM$ and we want to argue it is one of the $\fM_{k,a_p}$. Consider, first, any finite height $\phz$-module $\cM$ over $\cO_F$. It is canonically equipped with a meromorphic differential operator $N_{\nabla}$  satisfying the relation
\begin{equation*}
N_{\nabla} \circ \varphi = p {E(u)\over E(0)} \varphi \circ N_{\nabla}.
\end{equation*}
We say $\cM$ satisfies the {\em monodromy condition} provided $N_{\nabla}$ is without poles, which is equivalent to $N_{\nabla}$ being without a pole at $u = -p$ it turns out. In \cite{Kisin-FCrystals}, an equivalence $\cM \leftrightarrow D(\cM)$ is constructed between finite height $\phz$-modules over $\cO_F$ that satisfy the monodromy condition and effective filtered $\varphi$-modules. Returning to $\fM$, if $\cM=\fM\otimes_{\fS_\Lambda} \cO_F$ satisfies the monodromy condition (we abuse language and say $\fM$ itself satisfies the monodromy condition), then  $D(\cM)$ is  weakly-admissible. In practice, it is easy to determine if $D(\cM) = D_{\cris}^{\ast}(V_{k,a_p})$, and thus to calculate $\overline{V}_{k,a_p}$ from $\fM$. For $v_p(a_p) > \lfloor{{k-1\over p}\rfloor}$, this strategy can be enacted. We prove the following theorem.

\begin{thm}[Proposition \ref{prop:furtherdescent}]\label{thm:explicit-descent-intro}
Let $k\geq 2$ and suppose $v_p(a_p) > \lfloor{{k-1\over p}\rfloor}$ and $k \geq 2 p +1$. Then, there exists a polynomial $P \in \fm_F[u]$ of degree at most $k-1$ with $P(0) = a_p$ such that $\fM = \fS_{\Lambda}^{\oplus 2}$ equipped with $\varphi = \begin{smallpmatrix} P & -1 \\ E^{k-1} & 0 \end{smallpmatrix}$ satisfies the monodromy condition and $\fM \otimes_{\fS_{\Lambda}} \cO_F \cong \cM_{k,a_p}$.
\end{thm}

Theorem \ref{thm:intro} follows in weights $k \geq 2p+1$ since $\fM/\fm_F\fM$ is independent of $a_p$ (the theorem is known in small weights by prior work). We stress the content of Theorem \ref{thm:explicit-descent-intro} is entirely contained in finding an $\fM$ that satisfies the monodromy condition. The  polynomial $P$ in Theorem \ref{thm:explicit-descent-intro} is $p$-adically near to the truncation of $a_p(1+u^p/p)^{k-1}$ to degree $k-1$, which we note lies in $\mathfrak m_F[u]$ when $v_p(a_p) > \lfloor{{k-1\over p}\rfloor}$.

We end by describing the conceptual part of the strategy used to prove Theorem \ref{thm:explicit-descent-intro}. Since we first prove a more general statement for any $v_p(a_p) > 0$, we will ignore the issues of integrality and work over $\cO_F$. First, we determine the $\phz$-module $\cM_{k,0}$ corresponding to $a_p = 0$. This is one case where calculating using the definitions in \cite{Kisin-FCrystals} is accessible. In Section \ref{sec:family-2d}, we give a trivialization $\cM_{k,0} = \cO_F^{\oplus 2}$ in which $\varphi=\begin{smallpmatrix} 0 & -1 \\ E^{k-1} & 0 \end{smallpmatrix}$ and the monodromy operator $N_{\nabla,0}$, which has no poles, is completely explicit. Considering all operators $\varphi: F[\![u]\!]^{\oplus 2} \rightarrow F[\![u]\!]^{\oplus 2}$ satisfying 
$$
\varphi \circ N_{\nabla,0} = p{E(u)\over E(0)} N_{\nabla,0} \circ \varphi,
$$
there is a one-parameter family $\{\varphi_{a_p}\}$  with the simple form $\varphi_{a_p} = \begin{smallpmatrix} a_p \zeta & -1 \\ E^{k-1} & 0 \end{smallpmatrix}$ where $\zeta \in 1 + uF[\![u]\!]$ is an explicit series lying in the ring $R$ of functions on the closed disc $|u|_p \leq p^{-1/p}$. Via $\varphi_{a_p}$, we consider $R^{\oplus 2}$ as a $\varphi$-module $\widetilde{\cM}_{k,a_p}$ over $R$ with height $\leq k-1$, and we prove that we can descend $\widetilde{\cM}_{k,a_p}$ to a $\phz$-module $\cM$ over $\cO_F$, with features (except integrality) as in Theorem \ref{thm:explicit-descent-intro}. The crucial observation at this point is that such an $\cM$ {\em must} satisfy the monodromy condition:\ the canonical operator $N_{\nabla,\cM}$ associated with $\cM$ agrees with $N_{\nabla,0}$ after base change from $\cO_F$ to $R$ and so $N_{\nabla,\cM}$ has no pole at $u=-p$. After a short calculation, we conclude  $\cM \cong \cM_{k,a_p}$.

A significant portion of this article is devoted to an algorithm, and the attendant $p$-adic analysis, providing the descent from $R$ to $\cO_F$ described in the previous paragraph. The main mechanism is ``row reduction'' for semilinear operators. Related processes can be found in \cite{CarusoDavidMezard-Calculation, LLLM-ShapesShadows}, though those works focus on some more general aspects while simultaneously restricting to the small weight situations. 

\subsection{Acknowledgements} We would like to thank Laurent Berger and Tong Liu for helpful conversations related to this project.  The first author was partially supported by NSF award DMS-1402005. The second author was supported by a grant from the Simons Foundation/SFARI (\#585753).

\section{Kisin modules and the monodromy condition} \label{sec:kisinmodules}

For this section, we allow $K/\mathbb{Q}_p$ to be a general finite extension and work in any dimension; we will restrict to $K = \mathbb{Q}_p$ and dimension two starting in Section \ref{sec:family-2d}.   Here, we establish notations and the main theoretical $p$-adic Hodge theory results we need on Kisin modules and the monodromy condition.  The key result is a criterion (Corollary \ref{cor:monoverR}) for a $\phz$-module to satisfy the monodromy condition (it is based on \cite[Proposition 5.3]{LLLM-ShapesShadows}).

\subsection{Background}\label{subsec:background}

Let $k$ be a finite field, $W(k)$ the ring of Witt vectors over $k$ and $K_0 = W(k)[1/p]$. Choose a finite, totally ramified, extension $K/K_0$ and let $\overline{K}$ be an algebraic closure of $K$. Define $G_K = \Gal(\overline K / K)$. Write $K = K_0(\pi)$ where $\pi$ is a uniformizer in $K$, and let $E(u) \in K_0[u]$ be the Eisenstein polynomial for $\pi$. Choose elements $\pi_0,\pi_1,\pi_2,\dotsc$ in $\overline{K}$ such that $\pi_0=\pi$ and $\pi_{n+1}^p = \pi_n$ for all $n \geq 0$. The field $K_\infty$ is defined to be the compositum of the $K(\pi_n)$ in $\overline{K}$, and $G_\infty$ is defined to be $\Gal(\overline K/K_\infty)$. 

For $r>0$, we write $\Delta_{[0,p^{-r}]}$ for the $p$-adic disc of radius $p^{-r}$ over $K_0$ in a coordinate $u$ and $\Delta = \bigcup_{r} \Delta_{[0,p^{-r}]}$ for the open $p$-adic unit disc over $K_0$. The ring of rigid analytic functions on $\Delta_{[0,p^{-r}]}$ is denoted by $\mathcal O_{[0,p^{-r}]}$ and, likewise, $\mathcal O \subseteq K_0[\![u]\!]$ denotes the ring of rigid analytic functions on $\Delta$. We write $\fS = W(k)[\![u]\!]$, which is a subring of $\mathcal O_{[0,p^{-r}]}$ for any $r >0$. The ring $K_0[\![u]\!]$ is equipped with a unique operator $\varphi$ such that $\varphi(u) = u^p$ and $\varphi$ acts as a lift of Frobenius on $K_0$. The rings $\mathcal O$ and $\mathcal O_{[0,p^{-r}]}$ are $\varphi$-stable. In fact, $\varphi(\mathcal O_{[0,p^{-r}]}) \subseteq \mathcal O_{[0,p^{-r/p}]} \subseteq \mathcal O_{[0,p^{-r}]}$.

We also choose $F/\Qp$ a finite extension, which will play the role of linear coefficients. We assume that $F$ contains a subfield isomorphic to $K_0$. We write $\Lambda$ for the ring of integers in $F$ and $\F$ for the residue field. The notations of the previous paragraph extend, naturally. Specifically, $\mathcal O_{F,[0,p^{-r}]} = \mathcal O_{[0,p^{-r}]} \otimes_{\mathbb Q_p} F$ and $\mathcal O_F = \mathcal O \otimes_{\mathbb Q_p} F$, which is the ring of rigid analytic functions on $[K_0:\Qp]$-many open unit discs over $F$. Likewise, we define $\fS_{\Lambda} = \fS\otimes_{\mathbb Z_p} \Lambda$ and $\fS_F = \fS_{\Lambda}[1/p] \subseteq \mathcal O_F$. The action of $\varphi$ on $K_0[\![u]\!]$ extends to $(K_0\otimes_{\mathbb Q_p} F)[\![u]\!]$ linearly in $F$ and all the above rings are $\varphi$-stable.
 
Assume that $R \subseteq (K_0\otimes_{\mathbb Q_p} F)[\![u]\!]$ is a $\varphi$-stable subring containing $E$. A $\varphi$-module over $R$ is a finite free $R$-module $M$ equipped with an injective $\varphi$-semilinear operator $\varphi_M : M \rightarrow M$. We write $\Mod_{R}^{\varphi}$ for the category whose objects are $\varphi$-modules over $R$ and with morphisms being $R$-module morphisms that commute with $\varphi$. If $h \geq 0$, then an element $M \in \Mod_R^{\varphi}$ is said to have ($E$)-height $\leq h$ if the linearization $\varphi_M^{\ast}(M) = R\otimes_{\varphi,R} M \rightarrow M$ of $\varphi_M$ has cokernel annihilated by $E^h$. We write $\Mod_{R}^{\varphi,\leq h}\subseteq \Mod_{R}^{\varphi}$ for the full subcategory of $\varphi$-modules with height $\leq h$.

\begin{defn}
A \emph{Kisin module of height $\leq h$} over $\fS_{\Lambda}$ (resp.\ $\fS_{F}$) is an object in $\Mod_{\fS_{\Lambda}}^{\varphi,\leq h}$ (resp.\ $\Mod_{\fS_{F}}^{\varphi,\leq h}$).
\end{defn}

Though our ultimate aim is questions on crystalline Galois representations, for now we work with possibly non-trivial monodromy. Following \cite{Kisin-FCrystals}, let $\Mod_{\mathcal O_F}^{\varphi, N, \leq h}$ denote the category of triples $(\cM, \phz_{\cM}, N_{\cM})$ where $\cM \in \Mod_{\mathcal O_F}^{\varphi,\leq h}$ with Frobenius operator $\varphi_{\cM}$ and $N_{\cM}:\cM/u\cM \ra \cM /u\cM$ is a $K_0 \otimes_{\Qp} F$-linear endomorphism such that $N_{\cM} \phz_{\cM}|_{u=0} = p \phz_{\cM}|_{u=0} N_{\cM}$. Here and below $(-)|_{u=0}$ means to calculate modulo $u$. Similarly, we define $\Mod_{\fS_{\Lambda}}^{\varphi, N, \leq h}$ (resp.\ $\Mod_{\fS_{F}}^{\varphi, N, \leq h}$) as in \cite[(1.3.12)]{Kisin-FCrystals}. Note:\ even if $\fM$ is defined over $\fS_{\Lambda}$, we nevertheless take $N_{\fM}$ to be defined on $(\fM/u\fM) \otimes_{\Lambda} F$. Extension of scalars defines functors
\begin{equation}\label{eqn:functor-composition}
\Mod_{\fS_\Lambda}^{\varphi, N,\leq h} \rightarrow  \Mod_{\fS_F}^{\varphi, N,\leq h} \rightarrow \Mod_{\cO_F}^{\varphi, N, \leq h}.
\end{equation}
Below we will just write $\cM \in \Mod_{\cO_F}^{\varphi,N,\leq h}$ with the operators $\varphi_{\cM}$ and $N_{\cM}$ understood.

Let $\mathrm{MF}^{\varphi, N}_F$ denote the category of filtered $(\varphi, N)$-modules over $F$ (see \cite[Section 3.1.1]{BreuilMezard-Multiplicities}). Then, Kisin defined in \cite[(1.2.7-8)]{Kisin-FCrystals} a covariant functor $D : \Mod^{\phz, N, \leq h}_{\cO_F} \rightarrow \mathrm{MF}^{\varphi, N}_F$. The underlying vector space is $D(\cM) = \cM/u\cM$, the Frobenius on $D(\cM)$ is $\varphi_{\cM}|_{u=0}$, and the monodromy on $D(\cM)$ is $N_{\cM}$. The filtration, which is always effective and does not depend on $N_{\cM}$, is more involved. We will recall its definition in the proof of Corollary \ref{thm:descent-monodromy}. We also abuse notation and write $D$ for the composition of $D$ with any of the scalar extensions \eqref{eqn:functor-composition}.

\subsection{The monodromy condition}
We now discuss the monodromy condition, which cuts out a subcategory $\Mod_{\mathcal O_F}^{\varphi,N_{\nabla},\leq h} \subseteq \Mod_{\mathcal O_F}^{\varphi, N, \leq h}$ that is equivalent via $D(-)$ to the effective filtered $\varphi$-modules (\cite[Theorem 1.2.5]{Kisin-FCrystals}). Let $c_0 = E(0)$ and
\[
\lambda = \prod_{n  = 0}^{\infty} \phz^n(E/c_0) \in \cO_{F}.  
\]   
Define a derivation $N_{\nabla} = - u \lambda \frac{d}{du}$ on $\cO_{F}$. Recall that  $N_{\nabla} \phz = p(E/c_0)\phz N_{\nabla}$.

\begin{lemma}\label{lemma:mono-uniqueness}
Let $\mathcal M \in \Mod_{\cO_F}^{\varphi, N, \leq h}$. Assume that $\mathcal O_F[1/\lambda] \subseteq S \subseteq (K_0\otimes_{\Q_p} F)[\![u]\!]$ is stable under $\varphi$ and $N_{\nabla}$. Write $\mathcal M_S = \mathcal M \otimes_{\mathcal O_F} S$. Then, there exists a unique differential operator $N_{\nabla}^{\cM} : \cM_S \rightarrow \cM_S$ over $N_{\nabla}$ such that $N_{\nabla}^{\cM}|_{u=0}= N_{\cM}$ and  $N_{\nabla}^{\cM} \phz_{\cM} = p(E/c_0) \phz_{\cM} N^{\cM}_{\nabla}$.
\end{lemma}
\begin{proof}
The existence of $N_{\nabla}^{\cM}$ is \cite[Lemma 1.3.10]{Kisin-FCrystals}. We explain the (standard) argument for uniqueness. If $N_{\nabla, 1}^{\cM}$ and $N_{\nabla, 2}^{\cM}$ are two such operators, the difference $H  = N_{\nabla, 1}^{\cM} - N_{\nabla, 2}^{\cM}$ is an $S$-linear endomorphism of $\cM$ such that $H(\mathcal M_S) \subseteq u \mathcal M_S$ and
\begin{equation}\label{eqn:H-mono}
H \phz_{\cM} = p(E/c_0) \phz_{\cM} H.
\end{equation}
Since $\cM$ has height $\leq h$, and $E$ is a unit in $\cO_F[1/\lambda] \subseteq S$, $\cM_S$ is contained in the $S$-submodule generated by $\varphi_{\cM}(\cM_S)$. Thus, from \eqref{eqn:H-mono},  $H(\cM_S)$ is contained in the $S$-module generated by $\varphi_{\cM}(H(\cM_S))$. So, starting from $H(\cM_S) \subseteq u\cM_S$ we see by induction that $H(\cM_S) \subseteq u^{p^i}\cM_S$ for all $i$. Thus, $H(\cM_S) = \{0\}$.
\end{proof}

Given $\cM \in \Mod_{\cO_F}^{\varphi,N,\leq h}$, we typically write $N_{\nabla}^{\cM}$ for the differential operator on $\cM[1/\lambda]$ obtained from the previous lemma.

\begin{rmk}\label{rmk:monodrom-rmk}
For making matrix calculations, it is helpful to translate into matrix form. Choose a basis for $\cM$ over $\cO_F$ and write $C$ (resp.\ $B$) for the matrix of $\phz_{\cM}$ (resp. $N^{\cM}_{\nabla}$) whose column vectors record the action of the basis. {\em A priori}, $B$ has entries in $\cO_F[1/\lambda]$, but in fact $\lambda^{h-1}B$ has entries in $\cO_F$ (see the proof in \cite{Kisin-FCrystals}). The commutation relation for $N_{\nabla}^{\cM}$ and $\varphi_{\cM}$ is equivalent to
\begin{equation} \label{eq:matmono}
N_{\nabla}(C) + B C = p(E/c_0) C \phz (B).
\end{equation}
We sometimes refer to \eqref{eq:matmono} as the {\em monodromy relation}.
\end{rmk}

\begin{defn} \label{defn:Mcond}  If $\cM \in \Mod_{\cO_F}^{\varphi, N, \leq h}$, $\cM$ satisfies {\em the monodromy condition} if $N^{\cM}_{\nabla}(\cM) \subset \cM$.
\end{defn}
We will abuse language and also say $\fM \in \Mod_{\fS_{\Lambda}}^{\varphi, N, \leq h}$ (resp.\ $\widetilde{\fM} \in \Mod_{\fS_F}^{\varphi,N, \leq h}$) satisfies the monodromy condition if $\fM\otimes_{\fS_{\Lambda}} \cO_F$ (resp.\ $\widetilde{\fM}\otimes_{\fS_F} \cO_F$) satisfies the monodromy condition.


If $n \geq 0$, we write $\fS_{F,n}$ for the completion of $\fS_F$ at the ideal generated by $\phz^n(E)$, and we write $\iota_n : \mathcal O_F \iarrow \fS_{F,n}$ for the natural inclusion. For any embedding $\sigma:K_0 \iarrow F$, the roots of $\sigma(\phz^n(E))$ lie on $|u|=p^{-1/ep^n}$ and so the map $\iota_n$ factors through $\mathcal O_{F,[0,p^{-r}]}$ whenever $r \leq 1/ep^n$.  Given $\mathcal M \in \Mod_{\mathcal O_F}^{\phz,\leq h}$, we write $\mathcal M_n = \mathcal M \otimes_{\mathcal O_F,\iota_n} \fS_{F,n}$. By construction, $\iota_n(\lambda)$ is a unit multiple of $\phz^n(E)$ in $\fS_{F,n}$, so we also use $\iota_n$ to denote the natural map $\mathcal M[1/\lambda] \rightarrow \mathcal M_n[1/\phz^n(E)]$. The monodromy condition on $\cM$ is equivalent to $\iota_n N_{\nabla}^{\cM}(\cM) \subseteq \cM_n$ for all $n \geq 0$.
However, we have the following weaker criterion based on \cite[Proposition 5.3]{LLLM-ShapesShadows}.   
\begin{prop} 
\label{prop:Mcond}
If $\mathcal M \in \Mod_{\mathcal O_F}^{\phz, N, \leq h}$, then $\cM$ satisfies the monodromy condition if and only if $\iota_0N_{\nabla}^{\cM}(\cM) \subset \cM_0$.
\end{prop} 
\begin{proof}
One direction is clear. Supposing $\iota_0N_{\nabla}^{\cM}(\cM) \subset \cM_0$, we will prove by induction on $n$ that in fact $\iota_n N_{\nabla}^{\mathcal M}(\mathcal M) \subset \mathcal M_n$. So, let $n \geq 0$ and assume that $\iota_n N_{\nabla}^{\mathcal M}(\mathcal M) \subset \mathcal M_n$. Note that $\varphi$ induces a natural map $\phz: \fS_{F,n}  \rightarrow \fS_{F,n+1}$ and $\phz_{\mathcal M}$ induces a $\phz$-semilinear operator $\phz_n: \mathcal M_n \rightarrow \mathcal M_{n+1}$ such that the diagram
\begin{equation*}
\xymatrix{
\mathcal M \ar[r]^{\phz_{\mathcal M}} \ar[d]_{\iota_n} & \mathcal M \ar[d]^{\iota_{n+1}}\\
\mathcal M_n \ar[r]_-{\phz_n} & \mathcal M_{n+1}
}
\end{equation*}
commutes. Using $N_{\nabla}^{\cM}\varphi_{\cM} = p(E/c_0)\varphi_{\cM}N_{\nabla}^{\cM}$, we deduce
\begin{equation}\label{eqn:contain-monodromy-proof}
\iota_{n+1} N_{\nabla}^{\mathcal M}\phz_{\mathcal M}(\mathcal M) = {p\over c_0} \iota_{n+1}(E)\cdot  \phz_n \left( \iota_n N_{\nabla}^{\cM}(\mathcal M) \right) \subset \mathcal M_{n+1}.
\end{equation}
On the other hand, since $\mathcal M$ has height $\leq h$, $E^h\mathcal M$ is contained in the $\mathcal O_F$-span of $\phz_{\cM}(\cM)$. So,  \eqref{eqn:contain-monodromy-proof} implies that $\iota_{n+1}N_{\nabla}^{\mathcal M}(E^h\mathcal M) \subset \mathcal M_{n+1}$. The containment $\iota_{n+1}N_{\nabla}^{\mathcal M}(\mathcal M) \subset \mathcal M_{n+1}$ now follows from the Leibniz rule and the fact that $\iota_{n+1}(E)$ is a unit in $\fS_{F,n+1}$ for $n\geq 0$.     
\end{proof}


\begin{cor} \label{cor:monoverR} 
Let $\cM \in \Mod_{\cO_F}^{\varphi, N, \leq h}$, $r \leq 1/e$ and $\cM_r = \cM\otimes_{\cO_F} \mathcal O_{F,[0,p^{-r}]}$. The following are equivalent:
\begin{enumerate}[label=(\alph*)]
\item $\cM$ satisfies the monodromy condition.
\item There exists a differential operator $N_{\nabla}^{\cM_r}: \cM_r \rightarrow \cM_r$ over $N_{\nabla}$ such that $N_{\nabla}^{\cM_r}|_{u=0} = N_{\cM}$  and $N_{\nabla}^{\cM_r} \varphi_{\cM_r} = p(E/c_0) \varphi_{\cM_r} N_{\nabla}^{\cM_r}$.
\end{enumerate}
\end{cor}
\begin{proof}  
Clearly, (a) implies (b). Suppose we are given (b).  By Lemma \ref{lemma:mono-uniqueness}, $N_{\nabla}^{\cM} = N_{\nabla}^{\cM_r}$ on $\cM_r[1/\lambda]$ and so the assumption in (b)  forces $N_{\nabla}^{\cM}(\cM) \subseteq \cM_r \cap \cM[1/\lambda]$. On the other hand, since $r \leq 1/e$, the natural map $\cO_F \rightarrow \fS_{F,0}$ factors through $\mathcal O_{F,[0,p^{-r}]}$, from which we deduce $\iota_0 N_{\nabla}^{\cM}(\cM) \subseteq \cM_0$. So, we conclude (a) holds by Proposition \ref{prop:Mcond}.
\end{proof} 

\subsection{Kisin modules and Galois representations}

By \cite{Kisin-FCrystals}, Kisin modules over $\fS_{\Lambda}$ satisfying the monodromy condition are related to Galois representations. To be precise, denote by $\mathrm{MF}^{\varphi, N, \wa}_F \subseteq \mathrm{MF}_F^{\varphi, N}$ the full subcategory of weakly-admissible filtered $(\varphi, N)$-modules. Then, we have a contravariant equivalence of categories
$$
V_{\mathrm{st}}^{\ast} : \mathrm{MF}^{\varphi, N, \wa}_F \rightarrow \mathrm{Rep}_{F}^{\mathrm{st}}(G_K)
$$
where $\mathrm{Rep}_{F}^{\mathrm{st}}(G_K)$ is the category of $F$-linear semistable representations of $G_K$ (\cite[Section 3.1.2]{BreuilMezard-Multiplicities}).  Taking $N = 0$, this restricts to an equivalence $V_{\cris}^{\ast} : \mathrm{MF}^{\varphi, \wa}_F \rightarrow \mathrm{Rep}_{F}^{\cris}(G_K)$ onto the category of $F$-linear crystalline representation of $G_K$.

Let $\cO_{\cE, \Lambda}$ denote the $p$-adic completion of $\fS_{\Lambda}[1/u]$ and extend $\varphi$ from $\fS_{\Lambda}[1/u]$ to $\cO_{\cE,\Lambda}$ by continuity. Note that $\cO_{\cE, \Lambda} \otimes_{\Lambda} \F = \F(\!(u)\!)$. The category of \'etale $\phz$-modules over $\cO_{\cE, \Lambda}$ (resp. $\F(\!(u)\!)$) is denoted by $\Mod^{\phz, \et}_{\cO_{\cE, \Lambda}}$ (resp.  $\Mod^{\phz, \et}_{\F(\!(u)\!)}$).  By \cite{Fontaine-GFestschriftPaper}, there are contravariant equivalences of categories
\[
V_{\Lambda}^{\ast} :\Mod^{\phz, \et}_{\cO_{\cE, \Lambda}} \ra \mathrm{Rep}_{\Lambda} (G_{\infty}), \quad  V_{\F}^{\ast}:\Mod^{\phz, \et}_{\F(\!(u)\!)} \ra \mathrm{Rep}_{\F} (G_{\infty})
\]
that satisfy the compatibility 
\begin{equation}\label{eqn:modp-compat}
V^{\ast}_{\Lambda}(M) \otimes_{\Lambda} \F \cong V_{\F}^{\ast}(M \otimes_{\Lambda} \F)
\end{equation}
 for any $M \in \Mod_{\cO_{\cE,\Lambda}}^{\varphi,\et}$. In particular, if $\fM \in \Mod_{\fS_{\Lambda}}^{\varphi,N, \leq h}$ then we have $G_\infty$-representations $V_{\Lambda}^{\ast}(\fM \otimes_{\fS_{\Lambda}} \cO_{\cE,\Lambda})$ over $\Lambda$ and $V_{\F}^{\ast}(\fM\otimes_{\Lambda} \F[u^{-1}])$ over $\F$.
 
If $W$ is a representation of $G_K$, we use $W|_{G_\infty}$ denote $W$ as a $G_\infty$-representation via restriction.
   
\begin{thm}[Kisin] \label{thm:comp}  
 If $\fM \in \Mod_{\fS_\Lambda}^{\varphi, N, \leq h}$ and $\fM$ satisfies the monodromy condition, then $D(\fM)$ is weakly-admissible. Moreover, $V_{\Lambda}^{\ast}(\fM \otimes_{\fS_{\Lambda}} \cO_{\cE, \Lambda}) [1/p] \cong V_{\mathrm{st}}^{\ast}(D(\fM))|_{G_{\infty}}$.
\end{thm}
\begin{proof} This is a summary of results of \cite{Kisin-FCrystals}. Specifically, the first statement follows from applying Lemma 1.3.13, Lemma 1.3.10, and Theorem 1.3.8 of {\em loc.\ cit.}\ to $\fM$. The second statement follows from Corollary 2.1.4 and Proposition 2.1.5 in the same reference. (See also \cite[Theorem 5.4.1]{Liu-Torsion}.)
\end{proof}

If $W$ is an $\mathbb{F}$-linear representation of a group $G$, write $W^{\mathrm{ss}}$ for the semi-simplification of $W$ as a $G$-representation. If $V$ is an $F$-linear representation of $G_K$, we write $\overline{V}$ for $(T/\fm_F T)^{\mathrm{ss}}$ where $T \subseteq V$ is any $G_K$-stable lattice.

\begin{cor}\label{cor:kisin-mod-p}
Let $\fM \in \Mod_{\fS_{\Lambda}}^{\varphi, N, \leq h}$ and assume that $\mathfrak M$ satisfies the monodromy condition. Then, given a semi-simple $\mathbb F$-linear representation $V_{\mathbb F}$ of $G_K$ we have $
\overline{V_{\mathrm{st}}^{\ast}(D(\fM))} \cong V_{\mathbb F}$ if and only if $(V_{\mathbb F}^{\ast}(\mathfrak M \otimes_{\Lambda} \mathbb F[u^{-1}]))^{\mathrm{ss}} \cong V_{\mathbb F}|_{G_\infty}$.
\end{cor}
\begin{proof}
Recall, a semi-simple representation of $G_K$ in characteristic $p$ is tamely ramified (\cite[Proposition 4]{Serre-PropertiesGaloisiennes}). In particular, since $K_\infty/K$ is totally wildly ramified, if $W$ is a semi-simple representation of $G_K$ then $W|_{G_\infty}$ is semi-simple, and restriction of semi-simple representations of $G_K$ to semi-simple representations of $G_\infty$ is a fully faithful functor. Thus, by Theorem \ref{thm:comp} and \eqref{eqn:modp-compat} we have $\overline{V_{\mathrm{st}}^{\ast}(D(\mathfrak M))}|_{G_\infty} \cong (V_{\F}^{\ast}(\fM \otimes_{\Lambda} \F[u^{-1}]))^{\mathrm{ss}}$. The corollary follows.
\end{proof}

\section{A family of two-dimensional $\phz$-modules}\label{sec:family-2d}

From now on, we take $K = \Qp$ and restrict to the crystalline case by viewing $\Mod_{\cO_F}^{\varphi,\leq h}$ as a full subcategory of $\Mod_{\cO_F}^{\varphi,N,\leq h}$ by forcing $N_{\cM} = 0$.

We begin with some notation on two-dimensional $F$-linear crystalline representations of $G_{\mathbb Q_p}$. For each $a_p \in \mathfrak m_F$ and integer $h \geq 1$ there is a unique, up to isomorphism, $D_{h+1,a_p} \in \mathrm{MF}^{\varphi,\wa}_F$ such that $\varphi$ has characteristic polynomial $X^2 - a_pX + p^{h}$ and the filtration's non-trivial jumps are in degrees $0$ and $h$. Let $V_{h+1,a_p} = V_{\cris}^{\ast}(D_{h+1,a_p})$. Then, $V_{h+1,a_p}$ is an irreducible crystalline representation of $G_{\mathbb Q_p}$ with Hodge--Tate weights $0 < h$.\footnote{The convention here is that the cyclotomic character has Hodge--Tate weight $1$.} Every two-dimensional, irreducible, crystalline representation over $F$ is a twist of some such $V_{h+1,a_p}$. See \cite[Section 3.1.2]{BreuilMezard-Multiplicities} for details and references.

\begin{rmk}\label{rmk:basis-switch}
Typically, $D_{h+1,a_p}$ is presented as $Fe_1\oplus Fe_2$ where $Fe_1$ is the non-trivial line in the filtration on $D_{h+1,a_p}$ and the matrix of $\varphi$ in the basis $\{e_1,e_2\}$ is given by $\begin{smallpmatrix} 0 & -1 \\ p^{h} & a_p \end{smallpmatrix}$ (cf.\ \cite{BergerLiZhu-SmallSlopes,Breuil-SomeRepresentations2}). It is convenient for us, however, to use the basis $\{p^{h} e_2, -e_1\}$ in which the matrix of $\varphi$ is $\begin{smallpmatrix} a_p & -1 \\ p^{h} & 0 \end{smallpmatrix}$.
\end{rmk}

Our goal in this section is to associate to  $V_{h+1, a_p}$ an explicit finite height $\phz$-module $\widetilde{\cM}$ over $R = \cO_{F, [0, p^{-1/p}]}$ that satisfies condition (b) in Corollary \ref{cor:monoverR}.  We further explain (Theorem \ref{thm:descent-monodromy}) that any descent $\cM$  of  $\widetilde{\cM}$ to $\cO_F$ satisfies the monodromy condition and $D(\cM) \cong D_{h+1, a_p}$ is weakly-admissible.

From now on, we fix an integer $h \geq 1$. For $K = \Qp$ we use the uniformizer $\pi = -p$, so that $E(u) = u + p$.  Let $\fM_0=\fS_{\Lambda}^{\oplus 2}$ denote the Kisin module over $\fS_{\Lambda}$ with Frobenius $\phz$ given by 
\[
C_0 = \begin{pmatrix} 0 & -1 \\ E^{h} & 0 \end{pmatrix}.
\]
Clearly $\fM_0$ has height $\leq h$. Moreover, since $E(u) = u+p$, we have
$$
{C_0}|_{u=0} = \begin{pmatrix} 0 & -1 \\ p^h & 0 \end{pmatrix},
$$
which is the matrix of $\varphi$ acting on $D_{h+1,0}$ in the basis described in Remark \ref{rmk:basis-switch}. The key step in justifying $D(\fM_0) \cong D_{h+1,0}$ is showing $\fM_0$ satisfies the monodromy condition. We do that by explicitly determining the differential operator $N^{\cM_0}_{\nabla}$ on $\cM_0 = \fM_0 \otimes_{\fS_F} \cO_F$.

 Define $\lambda_+ = \prod_{n \geq 0} {\varphi^{2n}(E(u)/c_0)}$ and $\lambda_- = \prod_{n\geq 0} {\varphi^{2n+1}(E(u)/c_0)}$. Note the crucial identities:
\begin{equation} \label{eq:a0}
\lambda = \lambda_+\lambda_-, \quad \varphi(\lambda_+) = \lambda_-, \quad \varphi(\lambda_-) = {c_0\over E} \lambda_+ =: \lambda_{++}.
\end{equation}
For $f \in \cO_F$ we write $f ' = \frac{df}{du}$. Then, $\varphi(f)' = pu^{p-1} \varphi(f')$ and so from \eqref{eq:a0}, we deduce
\begin{equation}\label{eqn:a0-derivs}
\varphi(\lambda_+') = {1\over pu^{p-1}} \lambda_-', \quad \varphi(\lambda_-') = {c_0\over pu^{p-1}}\left({\lambda_+'\over E} - {\lambda_+\over E^2}\right) = {1\over pu^{p-1}} \lambda_{++}'.
\end{equation}

\begin{prop} \label{prop:diagonal}  In the natural basis for $\cM_0$, the matrix of $N^{\cM_0}_{\nabla}$ is
\[
B = \begin{pmatrix} h u\lambda_+ \lambda_-' & 0  \\ 0  & h u\lambda_-\lambda_+' \end{pmatrix}.
\]
In particular, $\cM_0$ satisfies the monodromy condition.  
\end{prop}
\begin{proof}
By uniqueness of $N_{\nabla}^{\cM_0}$, it suffices to confirm that the relation $N_{\nabla}(C) + BC = p(E/c_0)C\varphi(B)$ holds for $B$ and $C=C_0$ (Remark \ref{rmk:monodrom-rmk}). That is straightforward, using \eqref{eq:a0} and \eqref{eqn:a0-derivs}.
\end{proof}

\begin{rmk}
The base change of $\fM_0$ to the unramified quadratic extension of $\Qp$ is the direct sum of two Kisin modules of rank one. The monodromy condition can be checked after unramified base change, and rank one Kisin modules always satisfy the monodromy condition (\cite[Lemma 1.3.10(3)]{Kisin-FCrystals}), so  it is unsurprising that $\frak M_0$ satisfies the monodromy condition.
\end{rmk}

\begin{prop} \label{prop:familyoverR} 
For each $a_p \in F$, the matrix
\begin{equation}\label{eqn:phiap-matrix}
C_{a_p} := \begin{pmatrix} a_p \left(\lambda_- \over \lambda_{++}\right)^{h} & -1 \\ E^h & 0 \end{pmatrix}
\end{equation}
satisfies the monodromy relation \eqref{eq:matmono} with $B$ from Proposition \ref{prop:diagonal}.
\end{prop}

\begin{proof}
Consider $\zeta \in R$ and $Z = \begin{smallpmatrix} 1 & 0 \\ -\zeta & 1\end{smallpmatrix}$, so that $C := C_0 Z = \begin{smallpmatrix} \zeta & -1 \\ E^h & 0\end{smallpmatrix}$. We prove the stronger claim that the monodromy relation \eqref{eq:matmono} is satisfied by $C$ and $B$ if and only if $\zeta$ is an $F$-scalar multiple of $(\lambda_-/\lambda_{++})^h$. To see this, first note \eqref{eq:matmono} is equivalent to:
\begin{equation}\label{eqn:matmono-2}
0 = B C + N_{\nabla}(C) - {p\over c_0}E C \varphi(B) =  B C_0Z + C_0N_{\nabla}(Z) + N_{\nabla}(C_0)Z - {p\over c_0} E C_0 Z\varphi(B).
\end{equation}
Let $[-,-]$ be the usual matrix commutator. Then by Proposition \ref{prop:diagonal}, we have \eqref{eqn:matmono-2} is equivalent to
\begin{equation}\label{eqn:matmono-3}
0 = C_0\left({p\over c_0}E [\varphi(B), Z] + N_{\nabla}(Z)\right).
\end{equation}
Since $C_0$ is not a zero divisor in $M_2(F[\![u]\!])$, using \eqref{eqn:a0-derivs} it is straightforward to see \eqref{eqn:matmono-3} is equivalent to $\zeta$ being a solution to the differential equation
\begin{equation}\label{eqn:zeta-diffeq}
hu(E/c_0)(\lambda_{-}\lambda_{++}' - \lambda_{++}\lambda_{-}')\zeta + u\lambda \zeta' = 0.
\end{equation}
Since $\lambda = (E/c_0)\lambda_-\lambda_{++}$, the general solution to \eqref{eqn:zeta-diffeq}, in $F[\![u]\!]$, is given by $\zeta = a \left(\lambda_- /\lambda_{++}\right)^{h}$ with $a \in F$. This completes the proof.
\end{proof}

Let $p^{-2} < r < 1$. By definition of $\lambda_{++}$,  the matrix $C_{a_p}$ in \eqref{eqn:phiap-matrix} has entries in $\mathcal O_{F,[0,p^{-r}]}$. So, we may define $\widetilde{\cM}_{a_p} = \mathcal O_{F,[0,p^{-r}]}^{\oplus 2}$ as a $\varphi$-module (of height $\leq h$) over $\mathcal O_{F,[0,p^{-r}]}$ by declaring $\varphi$ acts in the natural basis of $\widetilde{\cM}_{a_p}$ via the matrix $C_{a_p}$. In this way, we view $\{\widetilde{\cM}_{a_p}\}$ as a family of $\varphi$-modules deforming $\cM_0$. An object $\mathcal M_{a_p} \in \Mod_{\cO_F}^{\varphi,\leq h}$ such that $\mathcal M_{a_p} \otimes_{\mathcal O_F} \mathcal O_{F,[0,p^{-r}]} \cong \widetilde{\cM}_{a_p}$ is called a {\em descent} of $\widetilde{\cM}_{a_p}$ to $\cO_F$. We use similar language to describe descents to $\fS_{\Lambda}$ and $\fS_F$. The purpose of Sections \ref{sec:descent-algorithm} and \ref{sec:application} is to show a descent (to $\fS_{F}$, even!) always exists for $a_p \in \fm_F$ and identify an exact condition on $v_p(a_p)$ under which $\widetilde {\cM}_{a_p}$ further descends to $\fS_{\Lambda}$. For now, we prove just the following result, which connects the family $\{\widetilde{\cM}_{a_p}\}$ to Galois representations.

\begin{thm}\label{thm:descent-monodromy}
Let $p^{-2} < r < 1$ and $\widetilde{\mathcal M}_{a_p}$ be as above. If $a_p \in \fm_F$ and $\cM_{a_p} \in \Mod_{\cO_F}^{\varphi,\leq h}$ is a descent of $\widetilde{\cM}_{a_p}$, then $\cM_{a_p}$ satisfies the monodromy condition, $D(\cM_{a_p})$ is weakly-admissible, and $V_{\cris}^{\ast}(D(\cM_{a_p})) = V_{h+1,a_p}$.
\end{thm}

\begin{proof}
Let $\cM=\cM_{a_p}$ be as in the statement. By Corollary \ref{cor:monoverR} and Proposition \ref{prop:familyoverR}, $\cM$ satisfies the monodromy condition. 

In order to justify the weak-admissibility of $D(\cM)$, we will have to explicitly calculate the filtration on $D(\mathcal M)$ as defined in \cite[(1.2.7)]{Kisin-FCrystals}. This is not so difficult, but we would like to mention that if $\cM_{a_p}$ were to descend to $\fS_{\Lambda}$, which is the most interesting case for us, then the weak-admissibility is automatic by Theorem \ref{thm:comp}.

Let $s$ be such that $p^{-1} < s < 1$ and $r \leq s$. Write $\cO_s = \cO_{F,[0,p^{-s})}$. Then $\mathcal O_{F,[0,p^{-r}]} \subseteq \cO_s$, so $\cM_{s} := \cM_{a_p}\otimes_{\cO_F} \cO_s$ has a basis $\{e_1,e_2\}$ in which $\varphi$ acts via $C_{a_p}$ in Proposition \ref{prop:familyoverR}. In particular,
\begin{equation*}
(1\otimes \varphi)(\varphi^{\ast}\cM_s) = \cO_s e_1 \oplus \cO_s E^h e_2 \subseteq \cM_s.
\end{equation*}
The left-hand side is equipped with a decreasing filtration, which in degrees $i \geq 0$ is given by
$$
\Fil^{i}\left((1\otimes \varphi)(\varphi^{\ast}\cM_s)\right) := (1\otimes \varphi)(\varphi^{\ast}\cM_s) \cap E^i \cM_s = \cO_s E^i e_1 \oplus \cO_s E^{\max\{h,i\}} e_2.
$$
Write $\xi: D(\cM)\otimes_{F} \cO_F \rightarrow \cM
$  for the map from \cite[Lemma 1.2.6]{Kisin-FCrystals}. Thus $\xi$ is injective, $\varphi$-equivariant, and the induced map $\xi_s : D(\cM)\otimes_F \cO_s \rightarrow \cM_s$ defines an isomorphism $\xi_s: D(\cM)\otimes_F \cO_s \cong (1\otimes \varphi)(\varphi^{\ast}\cM_s)$, inducing a filtration on $D(\cM) \otimes_F \cO_s$. Explicitly, if we choose $x_i \in D(\cM)\otimes_F \cO_s$ such that $\xi_s(x_1) = e_1$ and $\xi_s(x_2) = E^he_2$ then
\begin{equation}\label{eqn:filtrationDotimes}
\Fil^i(D(\cM)\otimes_F \cO_s) = \begin{cases}
D(\cM)\otimes_F \cO_s & \text{if $i \leq 0$;}\\
\cO_s E^i x_1 \oplus \cO_s x_2 & \text{if $1\leq i \leq h$;}\\
\cO_s E^i x_1 \oplus \cO_s E^{i-h} x_2 & \text{if $i > h$}.
\end{cases}
\end{equation}
The filtration $\Fil^i D(\cM)$ is then defined to be the image of $\Fil^i(D(\cM)\otimes_F \cO_s)$ under the map
$$
D(\cM)\otimes_F \cO_s \rightarrow D(\cM)\otimes_F \cO_s/E\cO_s \cong D(\cM).
$$
Write $\overline{x}\in D(\cM)$ for the image of $x \in D(\cM)\otimes_F \cO_s$ under the previous map. From \eqref{eqn:filtrationDotimes} we have
$$
\Fil^iD(\cM) = \begin{cases}
D(\cM) & \text{if $i \leq 0$;}\\
F \overline{x_2} & \text{if $1 \leq i \leq h$;}\\
(0) & \text{if $i > h$}.
\end{cases}
$$
Since $\xi$ is injective and commutes with $\varphi$, we have $\varphi(x_2) = -\varphi(E^h) x_1$. In particular,  the non-trivial line $F\overline{x_2}$ in the filtration on $D(\cM)$ is not $\varphi$-stable. Since $\varphi$ acting on $D(\cM)$ has characteristic polynomial $X^2 - a_p X + p^h$, it follows that $D(\cM)$ is weakly-admissible and $D(\cM) \cong D_{h+1,a_p}$. (One could also use Remark \ref{rmk:basis-switch}.) The final claim, that $V_{\cris}^{\ast}(D(\cM)) \cong V_{h+1,a_p}$, now follows from the discussion at the start of this section.
\end{proof}
\section{Descent algorithm}\label{sec:descent-algorithm}
The goal of this section is to explain an algorithm for descending from $R = \cO_{F, [0,p^{-r}]}$ to $\fS_F$.  The algorithm specifically will allow us to descend the $\phz$-module $\widetilde{\cM}_{a_p}$ defined at the end of Section \ref{sec:family-2d} to $\fS_F$, when $a_p \in \fm_F$, and even to $\fS_{\Lambda}$ when $v_p(a_p) \gg 0$.  It proceeds via ``row reduction'' for semilinear operators and is inspired by related processes that appear in \cite{CarusoDavidMezard-Calculation} and \cite[\S 4]{LLLM-ShapesShadows}.  In those settings, an integral structure of the attendant $\varphi$-modules is a given. The novelty here is that we begin over the larger ring $R$ where $p$ is inverted.  In order to arrive at a descent defined over $\fS_{\Lambda}$ (and thus calculate reductions of Galois representations; cf.\ Corollary \ref{cor:reduction}), we need to make a number of careful estimates as the algorithm is carried out, and we have thus chosen to present the algorithm in a generality where those estimates are most clear. It may also be helpful for future applications.

\subsection{Notations}\label{subsec:notations}

Choose $m > 1$ and write
\begin{equation*}
R = \mathcal O_{F,[0,p^{-1/m}]} = \left\{ f = \sum a_i u^i \in F[\![u]\!] \mid i + mv_p(a_i) \rightarrow \infty \text{ as $i \rightarrow \infty$}\right\}.
\end{equation*}
We equip $R$ with the valuation 
\begin{equation*}
v_R(f) = \min_i \{i + mv_p(a_i)\},
\end{equation*}
 which induces on $R$ the structure of an $F$-Banach algebra (\cite[Proposition 6.1.5/1]{BGR}). In particular, $R$ is complete for the $v_R$-adic topology. If $v$ is a real number, we define 
$$
H_v = \{f \in R \mid v_R(f) \geq v\}.
$$
Thus $H_v \subseteq R$ is an additive subgroup and $H_{v}H_{w} \subseteq H_{v+w}$ for any $v,w$. For $C \in M_2(R)$, if $C = (c_{ij})$ then we also define $v_R(C) = \min\{v_R(c_{ij})\}$. More specifically, we will also write
\begin{equation*}
C \in \begin{pmatrix} H_{v_{11}} & H_{v_{12}}\\ H_{v_{21}} & H_{v_{22}} \end{pmatrix}
\end{equation*}
with the obvious meaning. If we replace $H_{v_{ij}}$ by an asterisk $\ast$, then we mean no condition {\em a priori}.

We record the following interaction between  $v_R(-)$ and the Frobenius operator $\varphi: R \rightarrow R$.

\begin{lemma}\label{lem:frobenius-increase}
If $f \in H_v \cap u^jR$, then $\varphi(f) \in H_{j(p-1)+v} \cap u^{pj}R$.
\end{lemma}
\begin{proof}
Write $g = \sum a_i u^i$ so that $\varphi(g) = \sum a_i u^{ip}$. Then,
\begin{equation*}
v_R(\varphi(g)) = \inf_{i \geq 0} \left\{ip  + mv_p(a_{i})\right\} \geq \inf_{i\geq 0} \{i + mv_p(a_i)\} = v_R(g).
\end{equation*}
If $f = u^j g$, so that $\varphi(f) = u^{pj}\varphi(g)$, then
\begin{equation*}
v_R(\varphi(f)) = pj + v_R(\varphi(g)) \geq pj + v_R(g) = pj - j + v_R(f).
\end{equation*}
This completes the proof.
\end{proof}

For each $n \geq 0$, we define a truncation operator
\begin{align*}
T_{\leq n} : F[\![u]\!] &\longrightarrow F[u]\\
T_{\leq n}\left(\sum_{i=0}^{\infty} a_i u^i\right) &= \sum_{i=0}^n a_iu^i.
\end{align*}
We will use  analogous notations $T_{<n}$, $T_{\geq n}$, $T_{>n}$ for truncation of different types. We will frequently use that $T_{\ast}(H_v) \subseteq H_v$ for any truncation operator $T_{\ast}$ and any $v$.

\subsection{Analysis of certain row operations}\label{subsec:analysis-operations}
For this subsection, we fix non-negative integers $q,r,s,t$. Given $C\in M_2(R)$, we write
$$
T(C) = \begin{pmatrix} T_{\leq q}(c_{11}) & T_{\leq r}(c_{12}) \\ T_{\leq s}(c_{21}) & T_{\leq t}(c_{22}) \end{pmatrix}
$$
and define $(e_{ij}) = E(C) = C - T(C)$. Our goal is to study the behavior of $C\mapsto T(C)$ and $C \mapsto E(C)$ under certain operations of the form $A \ast_{\varphi} C := A C \varphi(A)^{-1}$ for $A \in \GL_2(R)$. We begin with a lemma.

\begin{lemma}\label{lemma:D-lemma}
Let $D \in M_2(R)$ be such that $D \in \begin{pmatrix} H_{r'+\gamma} & H_{r+\gamma} \\ H_{s+\gamma} & H_{s'+\gamma}\end{pmatrix}$ where $\gamma > 0$ and $r',s' \in \mathbb Q$ such that $r'+s' = r+s$.
\begin{enumerate}[label=(\alph*)]
\item If $n$ is a non-negative integer such that $n(p-1)+r'-s' \geq 0$ and $f \in H_{r'-s+\gamma'} \cap u^nR$, with $\gamma ' > 0$, then
$$
\begin{pmatrix} 1 & -f \\ 0 & 1 \end{pmatrix} \ast_{\varphi} D  - D \in \begin{pmatrix} H_{r'+\gamma+\gamma'} & H_{r+\gamma+\gamma'}\\ 0 & H_{r'+\gamma + n(p-1)+\gamma'} \end{pmatrix} \subseteq \begin{pmatrix} H_{r'+\gamma+\gamma'} & H_{r+\gamma+\gamma'}\\ 0 & H_{s'+\gamma+\gamma'} \end{pmatrix} .
$$\label{lemma-part:D-row}
\item If $g \in H_{\gamma'} \cap uR$, with $\gamma' > 0$, then
\begin{equation*}
\begin{pmatrix} 1 - g & 0 \\ 0 & 1 \end{pmatrix} \ast_{\varphi} D - D \in \begin{pmatrix} H_{r'+\gamma + \gamma'} & H_{r+\gamma + \gamma'}\\ H_{s+\gamma + \gamma' + p-1} & 0 \end{pmatrix}.
\end{equation*}\label{lemma-part:D-scale}
\end{enumerate}
\end{lemma}
\begin{proof}
First, since $f \in H_{r'-s+\gamma'}$, we have $f H_{s+\gamma} \subseteq H_{r'+\gamma+\gamma'}$, and $fH_{s'+\gamma} \subseteq H_{r+\gamma + \gamma'}$, the latter because $r'+s' = r+s$. Further, $\varphi(f) \in H_{r'-s+\gamma'+n(p-1)}$ by Lemma \ref{lem:frobenius-increase}. Since $r' + n(p-1) \geq s'$ we deduce $\varphi(f) \in H_{s'-s+\gamma'}$. So,
$$
\begin{pmatrix} 0 & -f\\ 0 & 0 \end{pmatrix}D\begin{pmatrix} 1 & \varphi(f) \\ 0 & 1 \end{pmatrix} \in \begin{pmatrix} H_{r'+\gamma + \gamma'} & H_{r+\gamma+\gamma'}\\ 0 & 0 \end{pmatrix}\begin{pmatrix} H_0 & H_{s'-s+\gamma'} \\ 0 & H_0 \end{pmatrix} \subseteq \begin{pmatrix} H_{r'+\gamma+\gamma'} & H_{r+\gamma+\gamma'}\\ 0 & 0 \end{pmatrix}.
$$
Returning to $\varphi(f) \in H_{r'-s+\gamma'+n(p-1)}$, it follows that
$$
D\begin{pmatrix}
0 & \varphi(f) \\ 0 & 0
\end{pmatrix} \in \begin{pmatrix} 0 & H_{2r'-s+\gamma+\gamma'+n(p-1)}\\0 & H_{r'+\gamma+\gamma'+n(p-1)}\end{pmatrix} \subseteq \begin{pmatrix} 0 & H_{r+\gamma+\gamma'}\\0 & H_{s'+\gamma+\gamma'}\end{pmatrix}.
$$
The containment \ref{lemma-part:D-row} now follows because
\begin{equation}\label{eqn:fconj-expression}
\begin{pmatrix} 1 & -f \\ 0 & 1 \end{pmatrix} \ast_{\varphi} D - D = \begin{pmatrix} 0 & -f\\ 0 & 0 \end{pmatrix}D \begin{pmatrix} 1 & \varphi(f) \\ 0 & 1 \end{pmatrix} + D\begin{pmatrix} 0 & \varphi(f) \\ 0 & 0 \end{pmatrix}.
\end{equation}

For \ref{lemma-part:D-scale}, the conjugation is first well-defined because $g \in H_{\gamma'}$ and $\gamma' > 0$. Moreover, $(1-\varphi(g))^{-1} = 1 + h$ where $h \in H_{p-1+\gamma'}$ (Lemma \ref{lem:frobenius-increase}). Then, the proof is as in \ref{lemma-part:D-row} except using 
\begin{equation*}
\begin{pmatrix} 1 - g & 0 \\ 0 & 1 \end{pmatrix} \ast_{\varphi} D - D = \begin{pmatrix} -g & 0 \\ 0 & 0 \end{pmatrix} D \begin{pmatrix} 1+h & 0 \\ 0 & 1 \end{pmatrix} + D \begin{pmatrix} h & 0 \\ 0 & 0 \end{pmatrix}
\end{equation*}
rather than \eqref{eqn:fconj-expression}.
\end{proof}

\begin{prop}\label{prop:operation-analysis}
Suppose that $\gamma > 0$, $c_r,c_s \in \Lambda^\times$ and $C\in M_2(R)$ such that
\begin{equation*}
C \in \begin{pmatrix} 0 & c_ru^r \\ c_su^s & 0 \end{pmatrix} + \begin{pmatrix} H_{r'+\gamma} & H_{r+\gamma} \\ H_{s+\gamma} & H_{s'+\gamma} \end{pmatrix}
\end{equation*}
where $r',s' \in \mathbb Q$ and $r'+s' = r+s$.

\begin{enumerate}[label=(\alph*)]
\item Assume $q \geq s + \max\{0,\lceil{{s'-r'\over p-1}\rceil}-1\}$ and let $n = q-s+1$. Set $v = v_R(e_{11})$. Then,
$$
\rho(C) = \begin{pmatrix} 1 & -e_{11}/c_su^s \\ 0 & 1 \end{pmatrix} \ast_{\varphi} C
$$
satisfies the following:
\begin{enumerate}[label=(\roman*)]
\item $\rho(C) \in  \begin{pmatrix} 0 & c_ru^r \\  c_su^s & 0 \end{pmatrix} + \begin{pmatrix} H_{r'+\gamma} & H_{r+\gamma} \\ H_{s+\gamma} & H_{s'+\gamma}\end{pmatrix}$;\label{lem:row-induction}
\item $T(\rho(C)) - T(C) \in \begin{pmatrix} H_{v+ \gamma} & \ast \\ \ast & H_{v + n(p-1)}\end{pmatrix}$;\label{lem:row-truncation}
\item $E(\rho(C)) \in \begin{pmatrix} H_{v+\gamma} & \ast \\ \ast & \ast \end{pmatrix} \cap \left(E(C) + \begin{pmatrix} \ast & H_{v+\gamma +r-r'} \\ 0 & H_{v+n(p-1)} \end{pmatrix} \right)$.\label{lem:row-error}
\end{enumerate}
\label{prop-part:row}
\item Set $v = v_R(e_{12})$. Then,
$$
\sigma(C) = \begin{pmatrix} 1 -e_{12}/c_ru^r& 0 \\ 0 & 1 \end{pmatrix} \ast_{\varphi} C
$$
satisfies the following:
\begin{enumerate}[label=(\roman*)]
\item $\sigma(C) \in  \begin{pmatrix} 0 & c_r u^r \\ c_s u^s & 0 \end{pmatrix} + \begin{pmatrix} H_{r'+\gamma} & H_{r+\gamma} \\ H_{s+\gamma} & H_{s'+\gamma}\end{pmatrix}$;\label{lem:scaling-induction}
\item $T(\sigma(C)) - T(C) \in \begin{pmatrix} H_{v+ \gamma+r'-r} & \ast \\ \ast & 0\end{pmatrix}$;\label{lem:scaling-truncation}
\item $E(\sigma(C)) \in \begin{pmatrix} \ast & H_{v+\gamma} \\ \ast & \ast \end{pmatrix} \cap \left( E(C) + \begin{pmatrix}  H_{v+\gamma+r'-r} & \ast \\ H_{v+p-1 + s-r} & 0 \end{pmatrix} \right)$.\label{lem:scaling-error}
\end{enumerate}\label{prop-part:scale}
\end{enumerate}
\end{prop}
\begin{proof}
The proof of either part is similar. We give complete details for \ref{prop-part:row} and less for \ref{prop-part:scale}.

Recall $e_{11} = T_{>q}(c_{11})$ and so $\rho(C)$ is well-defined because $q \geq s$. In fact, $e_{11} \in u^{q+1}R = u^{s+n}R$. For notation, let $f = e_{11}/c_su^s$. Since $v = v_R(e_{11}) \geq r'+\gamma > r'$, we can write $v = r' + \gamma'$ with $\gamma' \geq \gamma > 0$. Then, we have $f \in H_{v-s} \cap u^n R = H_{r'+\gamma'-s} \cap u^nR$ (remember $c_s$ is a constant unit). Since $n \geq \lceil{{s'-r'\over p-1}\rceil}$, we have $n(p-1) \geq s'-r'$. So, we are in position to apply Lemma \ref{lemma:D-lemma}\ref{lemma-part:D-row}.

Now write $C = \begin{smallpmatrix} 0 & c_ru^r \\ c_su^s & 0 \end{smallpmatrix} + D$ so that $D \in \begin{smallpmatrix} H_{r'+\gamma} & H_{r+\gamma} \\ H_{s+\gamma} & H_{s'+\gamma}\end{smallpmatrix}$. Writing $D' = \begin{smallpmatrix} 1 & -f \\ 0 & 1 \end{smallpmatrix} \ast_{\varphi} D$, we have
\begin{equation}\label{eqn:expandC'-row}
\rho(C) = 
\begin{pmatrix}
0 & c_ru^r \\ c_su^s & 0
\end{pmatrix}
+
\begin{pmatrix}
-e_{11} & -e_{11}\varphi(f) \\ 0 & c_su^s \varphi(f)
\end{pmatrix}
+ D + (D'-D)
\end{equation}
and Lemma \ref{lemma:D-lemma}\ref{lemma-part:D-row} implies that, because $r+\gamma + \gamma' = v + \gamma + r-r'$,
\begin{equation}\label{eqn:D-lemma-cons-row}
D'-D \in \begin{pmatrix} H_{v+\gamma} & H_{v+\gamma+r-r'}\\ 0 & H_{v+\gamma + n(p-1)} \end{pmatrix} \subseteq \begin{pmatrix} H_{r'+\gamma} & H_{r+\gamma}\\ 0 & H_{s'+\gamma}\end{pmatrix}.
\end{equation}
Moreover, $\varphi(f) \in H_{v-s+n(p-1)} \subseteq H_{s'-s+\gamma'}$ (by Lemma \ref{lem:frobenius-increase}) and so, since $\gamma'\geq \gamma$ and $s'-s = r-r'$, we have
\begin{equation}\label{eqn:expandC'-middle-row}
\begin{pmatrix}
-e_{11} & -e_{11}\varphi(f) \\ 0 & c_su^s \varphi(f)
\end{pmatrix} \in \begin{pmatrix} H_{v} & H_{v+\gamma + r-r'}\\ 0 & H_{v+n(p-1)} \end{pmatrix} \subseteq \begin{pmatrix} H_{r'+\gamma} & H_{r+\gamma} \\ 0 & H_{s'+\gamma} \end{pmatrix}.
\end{equation}
Thus, \ref{lem:row-induction} follows from \eqref{eqn:expandC'-row}, \eqref{eqn:D-lemma-cons-row}, \eqref{eqn:expandC'-middle-row}, and the assumption on $D$. Since $T(C) = \begin{smallpmatrix} 0 & c_ru^r \\ c_su^s & 0 \end{smallpmatrix} + T(D)$, from \eqref{eqn:expandC'-row} we see that
\begin{equation*}
T(\rho(C)) - T(C) = \begin{pmatrix} 0 & \ast \\ \ast & T_{\leq t}(c_su^s \varphi(f)) \end{pmatrix} + T(D'-D) \in \begin{pmatrix} H_{v+\gamma} & \ast \\ \ast & H_{v+n(p-1)}\end{pmatrix}
\end{equation*}
by \eqref{eqn:D-lemma-cons-row} and our previous estimate $\varphi(f) \in H_{v-s+n(p-1)}$. This proves conclusion \ref{lem:row-truncation}. Finally, note that $E(C) = E(D) = \begin{smallpmatrix} e_{11} & \ast \\ \ast & \ast\end{smallpmatrix}$. Thus we see, applying $E(-)$ to \eqref{eqn:expandC'-row}, that
\begin{equation*}
E(\rho(C)) \in\begin{pmatrix} 0 & \ast \\ \ast & \ast \end{pmatrix} + E(D-D') \in \begin{pmatrix} H_{v+\gamma} & \ast \\ \ast & \ast \end{pmatrix}.
\end{equation*}
This proves half of \ref{lem:scaling-error}, while 
$$
E(\rho(C)) \in E(C) + \begin{pmatrix}  \ast & H_{v+\gamma+r-r'} \\ 0 & H_{v+n(p-1)} \end{pmatrix}.
$$
follows from \eqref{eqn:D-lemma-cons-row} and \eqref{eqn:expandC'-middle-row}.

For part \ref{prop-part:scale}, let $g = e_{12}/c_ru^r \in H_{v-r} \cap uR = H_{\gamma'} \cap uR$, with $\gamma'=v-r \geq \gamma >0$. Define $h$ by $(1-\varphi(g))^{-1} = 1 + h$ as in the proof of Lemma \ref{lemma:D-lemma}\ref{lemma-part:D-scale}. Writing $C = \begin{smallpmatrix} 0 &  c_ru^r \\ c_su^s & 0 \end{smallpmatrix} + D$ and $D' = \begin{smallpmatrix} 1-g & 0 \\ 0 & 1 \end{smallpmatrix} \ast_{\varphi} D$, we have
\begin{equation}
\sigma(C) = \begin{pmatrix} 0 & c_ru^r \\ c_su^s & 0 \end{pmatrix} + \begin{pmatrix} 0 & -e_{12} \\ c_su^s h &  0 \end{pmatrix} + D + (D'-D).
\end{equation}
By assumption, $h \in H_{v+p-1-r}$. Since $v \geq r+\gamma$, we have $c_su^s h  \in H_{v+p-1+s-r} \subseteq H_{p-1+s + \gamma}$. So, part \ref{lem:scaling-induction} follows, using Lemma \ref{lemma:D-lemma}\ref{lemma-part:D-scale}. Statement \ref{lem:scaling-truncation} is trivial from the same lemma and that $r'+\gamma + \gamma' = v+\gamma + r'-r$. For \ref{lem:scaling-error}, the argument is as above.
\end{proof}

\subsection{Allowed operations and the descent theorem}\label{subsec:allowed-operations}
The previous subsection concerned two elementary operations, $\rho(-)$ and $\sigma(-)$, defined on $M_2(R)$. Here we apply that analysis to produce a criterion, Theorem \ref{theorem:descent-theorem}, for descending $\varphi$-modules from $R$ to a polynomial ring.

Fix non-negative integers $a$ and $b$, along with rational numbers $b' \geq a'$ such that $a + b = a' + b'$. We define $N=b$ if $b'=a'$, and otherwise
$$
N = b + \left\lceil {b'-a' \over p-1} \right\rceil - 1.
$$ 
Note that $N \geq b$ always. We now consider the specific truncation operation
$$
T(C) = \begin{pmatrix} T_{\leq N}(c_{11}) & T_{\leq a}(c_{12})\\
T_{\leq b}(c_{21}) & T_{\leq a}(c_{22})\end{pmatrix}
$$
on $M_2(R)$. As before, we define the error matrix $E(C)$ according to $C = T(C) + E(C)$.
\begin{defn}\label{defn:allowable}
Suppose $\gamma > 0$ and $c_a,c_b \in \Lambda^\times$. 
\begin{enumerate}[label=(\alph*)]
\item For $C \in M_2(R)$, we say  $C$ is $\gamma$-allowable with scalars $(c_a,c_b)$ if
\begin{equation*}
C \in \begin{pmatrix} 0 & c_a u^a \\ c_bu^b & 0 \end{pmatrix} + \begin{pmatrix} H_{a'+\gamma} & H_{a+\gamma} \\ H_{b+\gamma} & H_{b'+\gamma}\end{pmatrix}.
\end{equation*}
\label{defn-part:allowable-part}
\end{enumerate}
Now assume that $C$ is $\gamma$-allowable with scalars $(c_a,c_b)$.\footnote{We sometimes later omit the scalars and just say ``$\gamma$-allowable''.}
\begin{enumerate}[label=(\alph*)]
\setcounter{enumi}{1}
\item If $C$ is $\gamma$-allowable and $E(C) = (e_{ij})$ then we define
\begin{align*}
\varepsilon_{11} &= v_R(e_{11}) - a'; & \varepsilon_{12} &= v_R(e_{12}) - a;\\
\varepsilon_{21} &= v_R(e_{21}) - b; & \varepsilon_{22} &= v_R(e_{22}) - b'.
\end{align*}
The value $\varepsilon_C = \min\{\varepsilon_{ij}\}$ is called the error of $C$. (Note $\varepsilon_C \geq \gamma > 0$.)
\label{defn-part:allowable-error-values}
\item An allowed operation $C \mapsto \alpha(C)$ is one of the four operations 
\begin{align*}
\alpha_{11}(C) &:= \begin{pmatrix} 1 & -e_{11}/c_bu^b \\ 0 & 1 \end{pmatrix} \ast_{\varphi} C; & \alpha_{12}(C) &:= \begin{pmatrix} 1-e_{12}/c_au^a & 0 \\ 0 &  1\end{pmatrix} \ast_{\varphi} C;\\
\alpha_{21}(C) &:= \begin{pmatrix} 1 & 0 \\ 0 &  1-e_{21}/c_bu^b\end{pmatrix} \ast_{\varphi} C; & \alpha_{22}(C) &:= \begin{pmatrix} 1 & 0 \\ -e_{22}/c_au^a & 1 \end{pmatrix} \ast_{\varphi} C.
\end{align*}
(The operations $\alpha_{12}$ and $\alpha_{21}$ are well-defined by the geometric series.)
\label{defn-part:allowed-operations}
\end{enumerate}
\end{defn}

\begin{rmk}\label{rmk:allowed-analysis}
Each allowed operation is of the form $C \mapsto A\ast_{\varphi} C$ where $A = 1 + X$ with $X|_{u=0} = 0$ and $v_R(X) \geq \varepsilon_C+\min\{a'-b,b'-a\} = \varepsilon_C \pm (b'-a)$. Thus if $\varepsilon_C \geq |b'-a|$, then a finite composition of allowed operations is of the same form.
\end{rmk}

\begin{rmk}\label{rmk:allowed-oper-connection}
The allowed operations were {\em all} studied in Section \ref{subsec:allowed-operations}. Indeed, if $(i,j) = (1,\ast)$ then we set $(q,r,s,t,r',s') = (N,a,b,a,a',b')$ in Section \ref{subsec:analysis-operations}, in which case $\alpha_{11}(C) = \rho(C)$ and $\alpha_{12}=\sigma(C)$ as in Proposition \ref{prop:operation-analysis}. On the other hand, if $(i,j) = (2,\ast)$ then we set $(q,r,s,t,r',s') = (a,b,a,N,b',a')$ and so $\alpha_{22}(C) = \rho(C^{\circ})^{\circ}$ and $\alpha_{21}(C) = \sigma(C^{\circ})^{\circ}$, where $D\mapsto D^{\circ}$ is given by $D^{\circ} = \begin{smallpmatrix} 0 & 1 \\ 1 & 0 \end{smallpmatrix} \ast_{\varphi} D$. (That is, usual conjugation by $\begin{smallpmatrix} 0 & 1 \\ 1 & 0 \end{smallpmatrix}$.) 
\end{rmk}

\begin{lemma}\label{lemma:allowed-error-gains}
Suppose that $C$ is $\gamma$-allowable and fix $1 \leq i,j \leq 2$. Then, $C' = \alpha_{ij}(C)$ is $\gamma$-allowable. Moreover, writing $\varepsilon_{\ast}'$ for the entry-by-entry errors of $C'$ in Definition \ref{defn:allowable}\ref{defn-part:allowable-error-values}, we have:
\begin{enumerate}[label=(\alph*)]
\item $\varepsilon_{ij}' \geq \varepsilon_{ij} + \gamma$;
\item $\varepsilon_{k\ell}' \geq \min\{\varepsilon_{k\ell}, \varepsilon_{ij}+\min\{\gamma,p-1\}\}$ for any $(k,\ell)$, except if $(i,j) = (1,1)$ and $(k,\ell) = (2,2)$;
\item if $(i,j) = (1,1)$ then $\varepsilon_{22}' \geq \min\{\varepsilon_{22},\varepsilon_{11}\}$.
\end{enumerate}
In particular, $\varepsilon_{C'} \geq \varepsilon_C$.
\end{lemma}
\begin{proof}
Once one uses the translations in Remark \ref{rmk:allowed-oper-connection}, the $\gamma$-allowable assertion is contained in the conclusions labeled (i) in Proposition \ref{prop:operation-analysis} and the estimates are contained in the conclusions labeled \ref{lem:row-error} in Proposition \ref{prop:operation-analysis}. We detail the cases where $(i,j) = (1,1)$ or $(2,2)$, to highlight the exceptional asymmetry in (b).

For $(i,j) = (1,1)$, apply part \ref{prop-part:row} of Proposition \ref{prop:operation-analysis} to $C$ with $(q,r,s,t,r',s') = (N,a,b,a,a',b')$. Set $n = q - s + 1 = N-b+1$, so $n(p-1) + a'-b' \geq 0$. Then, Proposition \ref{prop:operation-analysis}\ref{prop-part:row}\ref{lem:row-error} gives, in terms of the $\varepsilon$'s,
\begin{align*}
\varepsilon_{11}' &\geq \varepsilon_{11} + \gamma; & 
\varepsilon_{12}' &\geq \min\{\varepsilon_{12},\varepsilon_{11} + \gamma\};\\
\varepsilon_{21}' &= \varepsilon_{21}; &
\varepsilon_{22}' & \geq \min\{\varepsilon_{22},\varepsilon_{11} + n(p-1) + a' - b'\}
\geq \min\{\varepsilon_{22},\varepsilon_{11}\},
\end{align*}
which implies the lemma if $(i,j)=(1,1)$. For $(i,j) = (2,2)$, apply part \ref{prop-part:row} of Proposition \ref{prop:operation-analysis} to $C^{\circ}$ (see Remark \ref{rmk:allowed-oper-connection}) with the parameters $(q,r,s,t,r',s') = (a,b,a,N,b',a')$. Then, $n = q-s+1 = 1$. By Proposition \ref{prop:operation-analysis}\ref{prop-part:row}\ref{lem:row-error} we get, in terms of the $\varepsilon$'s,
\begin{align*}
\varepsilon_{11}' &\geq \min\{\varepsilon_{11},\varepsilon_{22}+p-1 + b'-a'\} \geq \min\{\varepsilon_{11}, \varepsilon_{22}+p-1\} & 
\varepsilon_{12}' &= \varepsilon_{12}\\
\varepsilon_{21}' &\geq \min\{\varepsilon_{21}, \varepsilon_{22} + \gamma\}; &
\varepsilon_{22}' & \geq \varepsilon_{22} + \gamma.
\end{align*}
(In the estimate of $\varepsilon_{11}'$, we used $b'\geq a'$.) This completes the claim for $(i,j) = (2,2)$.
\end{proof}

\begin{prop}\label{prop:induction-step-descent-theorem}
Assume that $\gamma > 0$ and $C$ is $\gamma$-allowable with scalars $(c_a,c_b)$. Then, there exists a finite composition  $\alpha$ of allowed operations such that $C' = \alpha(C)$ satisfies the following properties:
\begin{enumerate}[label=(\alph*)]
\item $C'|_{u=0} = C|_{u=0}$;\label{prop-part:induction-u0}
\item $C'$ is $\gamma$-allowable with scalars $(c_a,c_b)$;\label{prop-part:induction-allowable}
\item $T(C') - T(C) \in \begin{smallpmatrix} H_{r} & \ast \\ \ast & H_r \end{smallpmatrix}$ where $r = \varepsilon_C + a' + \min\{\gamma,p-1\}$;\label{prop-part:induction-truncate}
\item $\varepsilon_{C'} \geq \varepsilon_C+\min\{\gamma,p-1\}$.\label{prop-part:induction-error} 
\end{enumerate}
\end{prop} 
\begin{proof}
For any composition $\alpha$,  \ref{prop-part:induction-u0} follows from Remark \ref{rmk:allowed-analysis} and \ref{prop-part:induction-allowable} follows from the conclusions \ref{lem:row-induction} in Proposition \ref{prop:operation-analysis}. For a single allowed operation, part \ref{prop-part:induction-truncate} follows from the conclusions \ref{lem:row-truncation} in Proposition \ref{prop:operation-analysis} (using the settings in Remark \ref{rmk:allowed-oper-connection}; recall that $b' \geq a'$ is assumed). The statement continues to hold for a composition of allowed operations because the error is non-decreasing after each operation by the final statement of Lemma \ref{lemma:allowed-error-gains}. 

So we only must show \ref{prop-part:induction-error} can be arranged. By Lemma \ref{lemma:allowed-error-gains}, we may repeatedly apply off-diagonal allowed operations to find a finite composition $\alpha$ of allowed operations such that $\widetilde{C} = \alpha(C)$ satisfies $\widetilde{\varepsilon}_{ij} \geq \varepsilon_C + \min\{\gamma,p-1\}$ for $i\neq j$. Then set $C' = \alpha_{22}\circ \alpha_{11}(\widetilde{C})$. From Lemma \ref{lemma:allowed-error-gains}, we have
\begin{equation*}
\varepsilon_{k\ell}'  \geq
\begin{cases}  \varepsilon_{\widetilde C} + \min\{\gamma,p-1\} & \text{if $k = \ell$};\\
\min\{\widetilde \varepsilon_{k\ell}, \varepsilon_{\widetilde C} + \min\{\gamma,p-1\}\},  & \text{if $k \neq \ell$}.
\end{cases}
\end{equation*}
Since $\varepsilon_{\widetilde  C} \geq \varepsilon_{C}$, by Lemma \ref{lemma:allowed-error-gains} again, this completes the proof of \ref{prop-part:induction-error}.
\end{proof}

\begin{rmk}\label{rmk:extra-useless}
The estimate in part \ref{prop-part:induction-truncate} of Proposition \ref{prop:induction-step-descent-theorem} can be strengthened though statement is more  complicated. Namely, we could have written that $T(C') - T(C) \in \begin{smallpmatrix} H_v & \ast \\ \ast & H_w \end{smallpmatrix}$ where  $v = \varepsilon_C + \min\{a'+\gamma,b'+p-1\}$ and $w = \varepsilon_C + \min\{b'+\gamma,a'+n(p-1)\}$ where $n =\left\lceil {b'-a' \over p-1} \right\rceil$ unless $a' = b'$,  then $n=1$. The same estimates could be used in part (c) of the next result as well.
\end{rmk}

\begin{thm}\label{theorem:descent-theorem}
Assume that $\gamma > 0$ and $C$ is $\gamma$-allowable. Then, there exists a matrix $A \in \GL_2(R)$ such that $C' := A\ast_{\varphi} C$ satisfies
\begin{enumerate}[label=(\alph*)]
\item $C|_{u=0} = C'|_{u=0}$,
\item $C' = T(C')$, and
\item $C'-T(C) \in \begin{smallpmatrix} H_r & \ast \\ \ast & H_r \end{smallpmatrix}$ where $r= \varepsilon_C + a' + \min\{\gamma,p-1\}$. 
\end{enumerate}
\end{thm}
\begin{proof}
Write $C=C^{(0)}$. Using Proposition \ref{prop:induction-step-descent-theorem}, we may for each $m>0$ choose a finite composition of allowed operations, say with matrix $A_m$, such that ${C^{(m)}}$ defined by
$$
C^{(m)} = A_m \ast_{\varphi} C^{(m-1)}
$$
satisfies the properties:
\begin{enumerate}
\item $C^{(m)}|_{u=0} = C|_{u=0}$,
\item $C^{(m)}$ is $\gamma$-allowable,
\item $T(C^{(m)})-T(C^{(m-1)}) \in \begin{smallpmatrix} H_{r_m} & \ast \\ \ast & H_{r_m} \end{smallpmatrix}$ where $r_m =  \varepsilon_{C^{(m-1)}} + a' + \min\{\gamma,p-1\}$, and
\item $\varepsilon_{C^{(m)}} \geq \varepsilon_{C^{(m-1)}}+\min\{\gamma,p-1\}$.
\end{enumerate}
For $m$ sufficiently large, $\varepsilon_{C^{(m)}} \geq |b'-a|$. In that case, Remark \ref{rmk:allowed-analysis} implies  $v_R(1-A_m) \geq \varepsilon_{C^{(m)}} \pm (b'-a)$ (for a constant $\pm$). Thus $A_m \rightarrow 1$ as $m \rightarrow \infty$,  meaning the infinite product $A := \prod_m A_m$ converges in $\GL_2(R)$. By induction again, $C' = A\ast_{\varphi} C$ satisfies the conclusion of the theorem.
\end{proof}

\section{Application}\label{sec:application}

We now specialize to the notations of Section \ref{subsec:notations} with $m = p$. So, we let $R = \cO_{F, [0, p^{-1/p}]}$.   Recall that just before Theorem \ref{thm:descent-monodromy}, for any $a_p \in \fm_F$, we defined a $\phz$-module $\widetilde{\cM}_{a_p}= R^{\oplus 2}$ with Frobenius given by 
\[
C_{a_p} := \begin{pmatrix} a_p \left(\lambda_- \over \lambda_{++}\right)^{h} & -1 \\ E^h & 0 \end{pmatrix},
\] 
where $E = u+p$ and
\begin{equation}\label{eqn:lambda-defn-reminder}
\lambda_- = \prod_{i \geq 0} \left(1 + {u^{p^{1+2i}}\over p}\right) = 1 + {u^p\over p} + \dotsb \quad \text{and} \quad \lambda_{++} = \prod_{i \geq 1} \left(1+ {u^{p^{2i}} \over p}\right) = 1 + {u^{p^{2}}\over p} + \dotsb.
\end{equation}
Our goal is to descend $\widetilde{\cM}_{a_p}$ from $R$ to $\fS_{F}$ and, when $v_p(a_p)$ is large enough, to descend it to $\fS_{\Lambda}$.  The first goal is carried out in Theorem \ref{thm:descent2} by applying the algorithm from Section \ref{sec:descent-algorithm}.  We then show, in Proposition \ref{prop:furtherdescent}, that an integral descent is exists when $v_p(a_p)$ is large enough.

\subsection{Preliminaries}\label{subsec:initial-descent}
We begin with some straightforward calculations.

\begin{lemma}  \label{lemma:est}
With $R = \mathcal O_{F,[0,p^{-1/p}]}$, we have
\begin{enumerate}[label=(\alph*)]
\item $v_R(\lambda_-) = v_R(\lambda_{++}) = 0$;
\item $v_R(1 - \lambda_{++})  = p^2 - p$;
\item $v_R(1-\varphi(\lambda_{++})) = p^3 - p^2$.
\end{enumerate}
\end{lemma}
\begin{proof} 
Part (a) is clear. For (b) we have $1-\lambda_{++} = \varphi(1-\lambda_-)$. Since $1-\lambda_-$ vanishes to order $p$ at $u = 0$, $v_R(1-\lambda_{++}) \geq p(p-1)$ by Lemma \ref{lem:frobenius-increase} and part (a). On the other hand, by definition $v_R(1-\lambda_{++}) \leq v_R(u^{p^2}/ p) = p^2 - p$ and this proves (b). Part (c) is proven similarly.
\end{proof}

\begin{lemma} \label{lem:integral} If $Q \in F[u]$ is of degree at most $d$ and $v_R(Q) > d$, then $Q \in \fm_F[u]$.    
\end{lemma} 
\begin{proof}
Clear.
\end{proof}

\begin{lemma} \label{lem:L1}  If $v_p(a_p) > \left\lfloor {h \over p} \right\rfloor$, then $T_{\leq h} \left(a_p\left({\lambda_- \over \lambda_{++} } \right)^h\right) \in \fm_F[u]$.
\end{lemma} 
\begin{proof}
Since $v_p(a_p) > \left\lfloor {h \over p} \right\rfloor$, we have by direct examination that
\begin{equation}\label{eqn:poly-estimate-first}
T_{\leq h} \left(a_p \left(1+{u^p \over p}\right)^h\right) \in \fm_F[u].
\end{equation}
Now, let $z = 1 - \lambda_{++}$ and $y = \varphi(\lambda_{++})-1$, so that 
\[
{\lambda_- \over \lambda_{++} } =\left(1+{u^p \over p}\right) (1 + y)\sum_{i=0}^{\infty} z^i.
\]  
By Lemma \ref{lemma:est}, $v_R(z) = p^2 - p$ and $v_R(y) = p^3 - p^2$,    Hence,
\begin{equation}\label{eqn:funny-diff-estimate}
v_R\left ( \left({\lambda_- \over \lambda_{++} }\right)^h - \left(1+{u^p \over p}\right)^h \right) \geq p^2 - p \geq p-1.  
\end{equation}
Since $p\left\lfloor {h \over p}\right\rfloor + p -1 \geq h$
and $v(a_p) > \lfloor{{h\over p}\rfloor}$, 
we have $p v_p(a_p) + p -1 > h$. So by \eqref{eqn:funny-diff-estimate} we conclude
\begin{equation}\label{eqn:compare-big-h}
v_R\left(a_p\left({\lambda_- \over \lambda_{++} }\right)^h - a_p\left(1+{u^p \over p}\right)^h\right) > h.
\end{equation}
The lemma now follows from  \eqref{eqn:poly-estimate-first}, \eqref{eqn:compare-big-h}, and Lemma \ref{lem:integral}.
%
\end{proof}

%

\subsection{Reductions}\label{subsec:reductions}

In this section, we prove  the main result on descent:

\begin{thm} \label{thm:descent2} 
Let $a_p \in \fm_F$.  Choose any rational number $a' \leq h/2$ such that $p v_p(a_p) > a'$. Define $N = h$ if $a' = h/2$, otherwise set
$
N = h + \left\lceil {h - 2 a' \over p -1 } \right\rceil - 1.
$
Then, there exists a descent $\widetilde{\fM}_{a_p}$ of $\widetilde{\cM}_{a_p}$ to $\fS_F$ such that the Frobenius on $\widetilde{\fM}_{a_p}$ is given by
\begin{equation*}
C = \begin{pmatrix} P & -1 \\ E^h & 0 \end{pmatrix}
\end{equation*}
where $P$ is a polynomial of degree $\leq N$ satisfying $P(0) = a_p$ and
\begin{equation}\label{eqn:P0-truncation}
v_R\left(P - T_{\leq N}\left(a_p \left({\lambda_- \over \lambda_{++} } \right)^h \right)\right) \geq p v_p(a_p) + \min\{pv_p(a_p) - a', p-1\}.
\end{equation}
Moreover, $\widetilde{\mathfrak M}_{a_p}$ satisfies the monodromy condition, $D(\widetilde{\mathfrak M}_{a_p})$ is weakly-admissible, and $V_{\cris}^{\ast}(D(\widetilde{\mathfrak M}_{a_p})) = V_{h+1,a_p}$.
\end{thm}

\begin{proof}
We choose $a = 0, b  = h$, let $a'$ be as in the theorem, and set $b' = h -a'$ in the setup of Section \ref{subsec:allowed-operations}. Then, $N$ is taken as in the statement of this theorem. 

By Lemma \ref{lemma:est}, $v_R((\lambda_-/\lambda_{++})^h) = 0$ and thus $v_R\left(a_p\left({\lambda_- \over \lambda_{++} } \right)^h\right) = pv_p(a_p) > a'$; we also have $v_R(u^h-E^h) \geq h +  p -1$. Thus, $C_{a_p}$ is $\gamma$-allowable with scalars $(c_0,c_h) = (-1,1)$, for $\gamma = \min\{p v_p(a_p) - a', p-1\}$. The error $\varepsilon_{C_{a_p}}$ of $C_{a_p}$ satisfies 
$$
\varepsilon_{C_{a_p}} = v_R\left(T_{>N}\left(a_p\left({\lambda_- \over \lambda_{++} } \right)^h\right)\right) \geq pv_p(a_p) - a'.
$$
Applying Theorem \ref{theorem:descent-theorem} to $C_{a_p}$, we get a $\varphi$-conjugate $C=A\ast_{\varphi} C_{a_p}$ of the form
$$
C = \begin{pmatrix} P & x \\ f & y \end{pmatrix}
$$
with $P$ a polynomial of degree at most $N$, $f$ a polynomial of degree at most $h$, and $x,y$ constants. Moreover, part (a) of Theorem \ref{theorem:descent-theorem} implies $P(0) = a_p$, $x = -1$, $y = 0$, and $f(0) = E(0)^h$, and part (c) implies that $P$ satisfies
$$
v_R\left(P - T_{\leq N}\left(a_p \left({\lambda_- \over \lambda_{++} } \right)^h \right)\right) \geq p v_p(a) + \min\{pv_p(a_p) - a', p-1\}.
$$
Comparing the determinant of $C$ to the determinant of $A \ast_{\varphi} C_{a_p}$, we see $f = r E^h$ where $r \in R^\times$. So, $f$ is a polynomial of degree at most $h$, with a zero of order $h$ at $u=-p$, and $f(0) = E(0)^h$. By unique factorization in $F[\![u]\!]$ the only possibility is that $r = 1$. 

So, $\widetilde{\fM}_{a_p} = \fS_{F}^{\oplus 2}$ with Frobenius given by $C$ satisfies the first half of the theorem. To justify the ``moreover'' portion, apply Theorem \ref{thm:descent-monodromy} to $\widetilde{\fM}_{a_p}\otimes_{\fS_F} \cO_F$.
\end{proof}  

We now address the question of when $\widetilde{\fM}_{a_p}$ from Theorem \ref{thm:descent2} is defined over $\fS_{\Lambda}$.   This is a delicate question and can depend on the choice of $a'$.  

\begin{prop} \label{prop:furtherdescent}
Assume $v_p(a_p) > \left\lfloor {h \over p} \right\rfloor$ and $h \geq 2p$.   Then, there exists a descent $\fM_{a_p}$ of $\widetilde{\cM}_{a_p}$ to $\fS_{\Lambda}$ such that the matrix of Frobenius is given by
\[
\begin{pmatrix} P & -1 \\ E^h & 0 \end{pmatrix}
\] 
where $P \in \fm_{F}[u]$ is a polynomial of degree at most $h$ and $P(0) = a_p$. Moreover, ${\mathfrak M}_{a_p}$ satisfies the monodromy condition, $D({\mathfrak M}_{a_p})$ is weakly-admissible, and $V_{\cris}^{\ast}(D({\mathfrak M}_{a_p})) = V_{h+1,a_p}$.
\end{prop}
\begin{proof} 
Let $a' = {h \over 2} - {p-1 \over 2}$, and write $h = qp + \delta$ where $\delta \leq p-1$ and $q = \left\lfloor {h \over p} \right\rfloor$.  Then,
\[
a' =  {h \over 2} - {p-1 \over 2} \leq {pq \over 2}.
\]
Thus $pv_p(a_p) - a' > pq - a' \geq p$ since $q \geq 2$.

In particular, Theorem \ref{thm:descent2} applies with $a'$, and note we have shown $pv_p(a_p) - a' > p$. Thus, we conclude there is a matrix $\begin{smallpmatrix} P & -1 \\ E^h & 0 \end{smallpmatrix}$ for the Frobenius on $\widetilde{\mathcal M}_{a_p}$ where $P$ is a polynomial of degree $N = h$ and such that
\begin{equation}
v_{R}(P - T_{\leq h}(a_p(\lambda_-/\lambda_{++})^h) \geq pv_p(a_p) + p-1 > h.
\end{equation}
By Lemma \ref{lem:integral}, since $P$ has degree at most $h$, we have $P \in \mathfrak m_F[u]$ if and only if $T_{\leq h}(a_p(\lambda_-/\lambda_{++})^h) \in \mathfrak m_F[u]$. The latter is true by Lemma \ref{lem:L1}, so the proof is complete.
\end{proof}

\begin{cor} \label{cor:reduction}  If $v_p(a_p) > \left\lfloor {h \over p} \right\rfloor$, then $\overline{V}_{h+1, a_p} \cong \overline{V}_{h+1, 0}$.

 More precisely, let $\Q_{p^2}$ denote the quadratic unramified extension of $\Q_p$ and $\chi$ the quadratic unramified $\F$-valued character of $G_{\Q_{p^2}}$.  If $\omega_2$ is a niveau 2 fundamental inertial character of $G_{\Q_{p^2}}$, then 
\[
 \overline{V}_{h+1, a_p} \cong \Ind^{G_{\Qp}}_{G_{\Q_{p^2}}}(\omega_2^h \chi).
\] 
\end{cor} 
\begin{proof}
We may suppose $h \geq 2p$ by \cite[Th\'eor\`eme 3.2.1]{Berger-Modp-compat}. Let $\fM_{a_p}$ be the Kisin module as in Proposition \ref{prop:furtherdescent}.  By Corollary \ref{cor:kisin-mod-p}, $\overline{V}_{h+1,a_p}$ is determined by the $\varphi$-module $\fM_{a_p} \otimes_{\Lambda} \F$. Since the reduction $\fM_{a_p} \otimes_{\Lambda} \F$ has Frobenius given by $\begin{smallpmatrix} 0 & -1 \\ u^h & 0 \end{smallpmatrix}$, which does not depend on $a_p$ subject to  $v_p(a_p) > \left \lfloor {h \over p}  \right \rfloor$, we have $\overline{V}_{h+1, a_p} \cong  \overline{V}_{h+1, 0}$. An explicit description of $V_{h+1,0}$ (and thus $\overline{V}_{h+1,0}$) is given in \cite[Proposition 3.2]{Breuil-SomeRepresentations2}.
\end{proof}

\bibliography{reductions_bibliography}
\bibliographystyle{abbrv}

\end{document}